\documentclass[11pt]{amsart}
\usepackage{cite}
\usepackage[width=1.2\textwidth]{caption}
\usepackage{wrapfig, lipsum,booktabs}
\usepackage{graphicx}
\usepackage{amssymb}
\usepackage{bbold}
\usepackage{epstopdf}
\usepackage{enumitem} 
\usepackage{verbatim}
\usepackage{bm}
\usepackage{multicol}
\usepackage{subcaption} 
\usepackage{fancyhdr} 
\usepackage{float}
\usepackage{hhline}
\usepackage{stmaryrd}
\usepackage{scalerel} 
\usepackage{color}
\usepackage{soul}
\usepackage[longtable]{multirow}
\usepackage{longtable}
\usepackage{tikz}
\usepackage{todonotes}
\usetikzlibrary{patterns}
\usepackage{pgfplots}
\usepackage{cancel}
\usepackage{amsmath}
\usepackage{mathtools}
\usepackage{calrsfs}
\DeclareMathAlphabet{\pazocal}{OMS}{zplm}{m}{n}
\newtheorem{theorem}{Theorem}[section]
\newtheorem{lemma}{Lemma}[section]

\newtheorem{remark}{Remark}[section]

\usepackage{algorithm} 
\usepackage[noend]{algpseudocode} 
\usepackage{xpatch}
\makeatletter
\xpatchcmd{\algorithmic}{\itemsep\z@}{\itemsep=0.4ex plus2pt}{}{}
\makeatother 
\usepackage{multirow}
\usepackage{amsmath,amssymb,eucal}
\usepackage{graphicx,epsfig}
\usepackage{psfrag}
\usepackage{url}
\usepackage[top=1in, bottom=1.in, left=1in, right=1in]{geometry}

\newcommand{\Th}{\mathcal{T}_h}

\newcommand{\Eh}{\mathcal{E}_h}

\setulcolor{red}

\newcommand{\poly}{\mathcal{P}}
\newcommand{\trans}{\mathrm{T}}
\newcommand{\CR}{\scaleto{CR}{5pt}}

\def\dunderline#1{\underline{\underline{#1}}}
\makeatletter
\renewcommand*\env@matrix[1][\arraystretch]{%
	\edef\arraystretch{#1}%
	\hskip -\arraycolsep
	\let\@ifnextchar\new@ifnextchar
	\array{*\c@MaxMatrixCols c}}
\makeatother

\pgfplotsset{compat=1.15}
\begin{document}
\title[$hp$-MG for $H$(div)-HDG]{$hp$-multigrid preconditioner for a divergence-conforming HDG scheme for the incompressible flow problems}

\author{Guosheng Fu}
\address{Department of Applied and Computational Mathematics and 
Statistics, University of Notre Dame, USA.}
\email{gfu@nd.edu}
\thanks{We gratefully acknowledge the partial support of this work
	by the U.S. National Science Foundation through grant DMS-2012031.}
\author{Wenzheng Kuang}
\address{Department of Applied and Computational Mathematics and 
	Statistics, University of Notre Dame, USA.}
\email{wkuang1@nd.edu}

\begin{abstract}
    In this study, we present an $hp$-multigrid preconditioner for
    a divergence-conforming HDG scheme for the generalized Stokes and the Navier-Stokes equations using an augmented Lagrangian formulation. 
    Our method relies on conforming simplicial meshes in two- and three-dimensions. The $hp$-multigrid algorithm is a multiplicative auxiliary space preconditioner that employs the lowest-order space as the auxiliary space, and we developed a geometric multigrid method as the auxiliary space solver. For the generalized Stokes problem, the crucial ingredient of the geometric multigrid method is the equivalence between the condensed lowest-order divergence-conforming HDG scheme and a Crouzeix-Raviart discretization with a pressure-robust treatment as introduced in \cite{linkNS2016}, which allows for the direct application of geometric multigrid theory on the Crouzeix-Raviart discretization. 
    The numerical experiments demonstrate the robustness of the proposed $hp$-multigrid preconditioner with respect to mesh size and augmented Lagrangian parameter, with
    iteration counts insensitivity to polynomial order increase. Inspired by the works by Benzi and Olshanskii \cite{benziOlshanskii06} and Farrell et al. \cite{farrell2019augmented}, we further test the proposed preconditioner on the divergence-conforming HDG scheme for the Navier-Stokes equations. Numerical experiments show a mild increase in the iteration counts of the preconditioned GMRes solver with the rise in Reynolds number up to $10^3$.
\end{abstract}
\maketitle

\section{Introduction}
\label{sec:intro}
The development of pressure-robust numerical schemes for incompressible flow problems has emerged as an active research topic in recent years. In the continuous setting, when the source term of the Stokes and Navier-Stokes equations is changed by a gradient field, only the pressure solution is affected, while the velocity solution remains unchanged. The key objective of pressure-robust numerical schemes is to maintain this invariance, so that the obtained numerical velocity solution is exactly divergence-free and its error estimation is independent of the regularity of the pressure solution of the incompressible flow problems. We refer to \cite{LinkeReview2017} for a comprehensive review of the pressure-robust finite element methods.

Among the developed pressure-robust numerical schemes, divergence-conforming hybridizable discontinuous Galerkin ($H$(div)-HDG) methods 
 are an efficient variant of the divergence-conforming discontinuous Galerkin (DG) methods. By leveraging the hybridization technique, coupling between the degrees of freedom (DOFs) of the $H$(div)-HDG is reduced. Additionally, the $H$(div)-HDG schemes can be statically condensed into a global system where only DOFs on the mesh skeleton remain, resulting in a much smaller matrix to be solved. At the same time, the HDG schemes retain attractive features of DG schemes, such as $hp$-adaptivity, stable upwind discretization of the convection term, and the ability to handle unstructured meshes with hanging nodes. The grad-velocity-pressure formulation for the $H$(div)-HDG was first introduced by Cockburn and Sayas in \cite{cockburnHdivHDGStokes} for the Stokes flow, and later generalized to the Brinkman equation in \cite{fuQiuBrinkman}. Lehrenfeld and Sch\"oberl \cite{lehrenfeld2010hybrid, lehrenfeldHighOrder} proposed and analyzed a symmetric interior penalty formulation for the $H$(div)-HDG scheme for the Stokes and Navier-Stokes equations. Different from the aforementioned methods, Rhebergen and Wells \cite{rhebergenStokes} constructed $H$(div)-HDG scheme with facet unknowns for the pressure as Lagrange multipliers to ensure the numerical velocity solution is exactly divergence-free, and this scheme is then extended to the Navier-Stokes equations in \cite{rhebergenNS}.

However, constructing efficient solvers and preconditioners for large-scale simulations for the statically condensed systems of high-order HDG schemes remains a challenge and has attracted interest in recent years. For the $H$(div)-HDG scheme for the incompressible flow problems, most solvers are focused on block preconditioner for the saddle-point structure of the condensed $H$(div)-HDG for the Stokes flow \cite{rhebergenPrecond1, rhebergenPrecond2, FuKuangBlock}, seeking a robust approximation of the Schur complement.
In this study, we propose a robust $hp$-geometric multigrid preconditioner for the
the $H$(div)-HDG scheme for both the generalized Stokes and Navier-Stokes equations. 
Constructing geometric multigrid algorithms for the condensed HDG schemes poses a challenge in designing a stable intergrid transfer operator between different mesh levels. This difficulty arises because the global DOFs of the condensed systems exist only on the mesh skeleton, and the finite element spaces of these global unknowns on different mesh levels are non-nested. In \cite{Tan09}, some intuitive choices of the intergrid transfer operators were tested for Poisson's equation but proved to be unstable and not optimal by numerical expriments.

To tackle this problem, Cockburn et al. \cite{cockburnHDGMG} first introduced a ``two-level" V-cycle multigrid for the HDG scheme for Poisson's equation in the spirit of the auxiliary space preconditioning (ASP) technique, where they employed a continuous piece-wise linear finite element space on the same mesh as the auxiliary space. The residual of the condensed HDG system is projected onto the auxiliary space and there is no coarse-grid facet unknown space.
However, adopting such an approach for the $H$(div)-HDG scheme problems for incompressible flow problems could be challenging. A stable projection operator on to an inf-sup stable pair of auxiliary spaces is needed, and moreover there is no natural correspondent for the upwind HDG discretization of the convection term in the continuous finite element spaces. 
Lu et al. \cite{luHMGHDG} have recently proposed an approach that keeps the HDG discretization on the coarse grid and constructs a novel prolongation operator, where hierarchical continuous finite element spaces are used to link facet variable spaces on different mesh levels.
Lu et al. proved that the standard V-cycle algorithm exhibits $h$-robust convergence for the local HDG scheme (LDG-H) for Poisson's equation, and extended their findings to the single-face hybridizable (SFH), hybrid Raviart-Thomas (RT-H), and hybrid Brezzi-Douglas-Marini (BDM-H) schemes for the Stokes equation in a more recent study \cite{lu2023homogeneous}. To solve the saddle-point structure of the HDG scheme, they employed an augmented Lagrangian with parameter $\Delta t$ and iteratively solved it ({\it{outer}} iteration), with the standard V-cycle used as the solver of the condensed SPD system ({\it{inner}} iteration). However, while the fast convergence of the outer iteration requires a large $\Delta t$, the V-cycle of the inner iteration is not robust with $\Delta t$ and may result in a large total iteration count.

In this study, we propose an $hp$-multigrid preconditioner for the condensed global systems of the grad-velocity-pressure formulation for the $H$(div)-HDG scheme for the generalized Stokes and the Navier-Stokes equations on conforming simplicial meshes. The augmented Lagrangian Uzawa iteration method is used to solve the condensed $H$(div)-HDG schemes, and we aim to accelerate Krylov space solvers by the $hp$-multigrid preconditioner for the primal operator on global velocity spaces. Our $hp$-multigrid is essentially a multiplicative ASP, with lowest-order global velocity spaces as the auxiliary space and geometric $h$-multigrid method as the auxiliary space solver. For the generalized Stokes equation, the key to the geometric multigrid algorithm is the establishment of the equivalence between the condensed lowest order $H$(div)-HDG scheme and the nonconforming Crouzeix-Raviart (CR) discretization with a pressure-robust treatment, with both methods introduced in \cite{LinkeReview2017} as approaches to satisfy the exact divergence-free constraint in incompressible flow problems. We note that the equivalence between the CR discretization and the lowest-order Raviart-Thomas (RT) mixed method is well-known for Poisson's equation \cite{arnold1985mixed, marini1985inexpensive}. Motivated by this, we proved in our previous work \cite{fk2022optimal} the equivalence between the condensed lowest-order HDG scheme ({\sf HDG-P0}) and a (scaled) CR discretization for the generalized Stokes equation. In this work, we extend and prove the equivalence between the condensed lowest-order $H$(div)-HDG scheme and a pressure-robust treated CR discretization.
Fig.\ref{fig:hdgEquivCR} demonstrates the DOFs of the condensed {\sf HDG-P0}, condensed lowest-order $H$(div)-HDG and CR discretization.
Such equivalence allows us to directly employ the rich geometric multigrid theory for CR discretization.
Then we use the geometric multigrid as the building block of the $hp$-multigrid for the condensed higher-order $H$(div)-HDG schemes. Numerical experiments support the robustness of the preconditioner with respect to mesh size and the augmented Lagrangian parameter, with iteration counts insensitivity to polynomial order increase. Inspired by the works by Benzi and Olshanskii \cite{benziOlshanskii06}, and Farrell et al. \cite{farrell2019augmented}, we further test the developed $hp$-preconditioenr on the condensed $H$(div)-HDG scheme for the linearized Naiver-Stokes equation by Picard and Newton's method. Our numerical experiments demonstrate that the proposed preconditioner still works well, with a mild increase of the iteration counts of the preconditioned GMRes solver for Reynolds number as large as $10^3$.

\begin{figure}[ht]
    \centering
    \includegraphics[width=.94\textwidth]{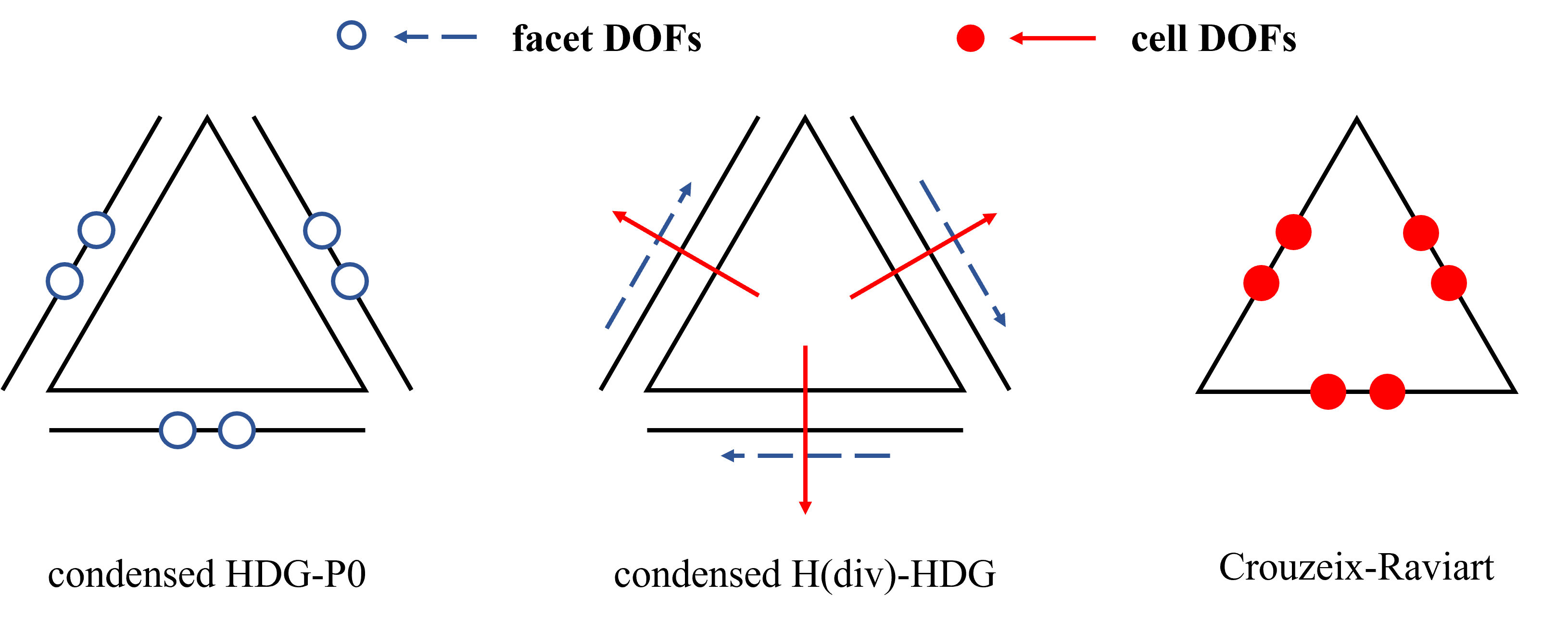}
    \caption{Comparison of the degrees of freedom of condensed HDG-P0, condensed lowest order $H$(div)-HDG and Crouzeix-Raviart schemes for incompressible flow problems.}
    \label{fig:hdgEquivCR}
\end{figure}

The rest of the paper is organized as follows. Basic notations and the finite element spaces to be used in the $H$(div)-HDG schemes are introduced in Section \ref{sec:notate}. In Section \ref{sec:generalStokes}, we present the grad-velocity-pressure formulation of the $H$(div)-HDG scheme for the generalized Stokes equation and propose the augmented Lagrangian Uzawa iteration for the condensed system. We first focus on the lowest order case, prove the equivalence to the CR discretization and present the geometric multigrid algorithm, then we use it as the building block for the $hp$-multigrid preconditioner for higher-order cases. In Section \ref{sec:ns}, we present the $H$(div)-HDG scheme for the Navier-Stokes equations. We use Picard linearization to search for a close enough initial guess and then use Newton's method to accelerate convergence. Two- and three-dimensional numerical experiments are performed in Section \ref{sec:num} and we conclude in Section \ref{sec:conclude}.

\section{Preliminaries and notations}
\label{sec:notate}
We assume a bounded polygonal/polyhedral domain $\Omega\in\mathbb{R}^d$, $d\in\{2, 3\}$, with boundary $\partial \Omega$. We denote $\Th$ as a conforming, shape-regular and quasi-uniform simplicial triangulation of the domain $\Omega$, and $\Eh$ as the set of facets of $\Th$. $\Eh$ is also referred to as \textit{mesh skeleton}. We split $\Eh$ into the boundary part $\Eh^\partial := \{F\in\Eh | \;F\subset\partial\Omega\}$ and the interior part $\Eh^o:=\Eh\backslash \Eh^\partial$. 
For each element $K\in\Th$ with element boundary $\partial K$, we denote $|K|$ as the measure of $K$, $(\cdot, \cdot)_K$ as the $L^2$-inner product over $K$, $\langle\cdot, \cdot\rangle_{\partial K}$ as the $L^2$-inner product over $\partial K$. We further define $(\cdot, \cdot)_{\Th}:= \sum_{K\in\Th}(\cdot, \cdot)_K$ as the discrete $L^2$ inner product over the domain, and $\langle\cdot, \cdot\rangle_{\partial\Th}:= \sum_{K\in\Th}\langle\cdot, \cdot\rangle_{\partial K}$ as the discrete $L^2$ inner product over all element boundaries.
Furthermore, we denote $\underline{n}_K$ as the unit normal vector on $\partial K$ pointing outward, $\mathsf{tng}(\underline{w}) := \underline{w}-(\underline{w}\cdot\underline{n}_K)\underline{n}_K$ as the tangential component and $\mathsf{nrm}(\underline{w}) := (\underline{w}\cdot\underline{n}_K)\underline{n}_K$ as the normal component on $\partial K$ of a vector field $\underline{w}$. 


As usual, we denote $\|\cdot\|_{p,S}$ and $|\cdot|_{p,S}$ as the norm and semi-norm of the Sobolev spaces $H^{p}(S)$ for the domain $S\subset \mathbb{R}^d$, with $p$ omitted when $p=2$, and $S$ omitted when $S=\Omega$ is the whole domain. 

For the finite element spaces, we denote $\poly^k(S)$ and $\widetilde{\poly^k}(S)$ as the polynomial and homogeneous polynomial of order $k$ over a simplex $S$. The following spaces are used to construct the $H$(div)-conforming HDG scheme and the multigrid method:
\begin{alignat*}{2}
    &W_h^k :=&& \{
        q_h \in L_2(\Omega):\; q_h|_K \in \poly^k(K),\; \forall K\in\mathcal{T}_h \},
    \\
    &W_{h,0}^k :=&& \{
        q_h \in W_h^k:\; \int_{\Omega}q_h\mathrm{dx} = 0\},
    \\
    &\underline{V}_h^k :=&& \{
        v_h \in H(div;\Omega):\; v_h|_K \in [\poly^k(K)]^d \oplus \underline{x}\;\widetilde{\poly^k}(K),\; \forall K\in\mathcal{T}_h\},
    \\
    &\underline{V}_{h,0}^k :=&& \{
        v_h \in \underline{V}_h^k:\; v_h\cdot\underline{n}|_F = 0, \;\forall F\subset\partial\Omega\},
    \\
    &\widehat{\underline{V}}_{h,k} :=&& \{
        \widehat{v}_h \in L_2(\mathcal{E}_h):\; \widehat{v}_h|_F \in \poly^k(F) \text{ and } \widehat{v}_h\cdot\underline{n}|_F = 0, \; \forall F\in\mathcal{E}_h\},
    \\
    &\widehat{\underline{V}}_{h,0}^k :=&& \{
        \widehat{v}_{h} \in \widehat{\underline{V}}_h^k:\; \mathsf{tng}(\widehat{v}_h)|_{F} = 0,\; \forall F\subset\partial\Omega\},
    \\
    &V_h^{\CR} :=&& \{v_h \in L^2(\Omega):\; v_h|_K \in \poly^1(K),\; \forall K\in\mathcal{T}_h, 
    \\& &&\text{ and $v_h$ is continuous at the barycenter $m_F$ of $F$, $\forall F\in\Eh^o$}
        \},
    \\
    &V_{h,0}^{\CR} :=&& \{v_h \in V_h^{\CR}:\; \text{$v_h(m_F) = 0$, $\forall F\in\Eh^\partial$}\}.
\end{alignat*}
$\underline{V}_h^k$ is the $k$-th order RT space \cite{RTspace1977} and $V_h^{\CR}$ is the first-order nonconforming CR space \cite{CR1973}. In this work, we use underline to denote the vector version of the finite element spaces, and double underline line to denote the matrix version, e.g., $\underline{V}_h^{\CR} := [V_h^{\CR}]^d$ and $\dunderline{W}_h^{k} := [W_h^{k}]^{d\times d}$.

Next, to facilitate our analysis of static condensation of the $H$(div)-conforming HDG scheme, we perform a hierarchical basis splitting for the RT space $\underline{V}_h^k$ following \cite{zag06}. We first divide $\underline{V}_h^k$ into \textit{global} subspace $\underline{V}_h^{k,\partial}$ and \textit{local} subspace $\underline{V}_h^{k,o}$, i.e.
\[
    \underline{V}_h^k=
    \underline{V}_h^{k,\partial}\oplus\underline{V}_h^{k,o}.
\]
For the global subspace $\underline{V}_h^{k,\partial}$, we have
\begin{align*}
    \underline{V}_h^{k,\partial}&:=
    \bigoplus_{F\in\Eh}\mathrm{span}\{\phi^0_F\}
    \;\;\oplus
    \bigoplus_{\begin{subarray}{c}
            F\in\Eh\\{i=1,\dots,n_F^k}
        \end{subarray}}\mathrm{span}\{\phi_F^i\},
\end{align*}
and
\begin{align*}
    \nabla\cdot\underline{v}_h^{\partial}|_K \in \poly^0(K),
    \quad (\underline{v}_h^{\partial}\cdot\underline{n}_K)|_F \in\poly^k(F), 
    \quad \forall F\in\partial K, K\in\Th, \quad \forall \underline{v}_h^{\partial}\in\underline{V}_h^{k,\partial},
\end{align*}
where $\phi_F^0$ is the basis of the lowest-order Raviart-Thomas ($\mathrm{RT0}$) space on the facet $F$, and $\phi_F^i$ is the higher-order basis of the divergence-free facet bubbles with normal component only supported on the facet $F$.
For the local subspace $\underline{V}_h^{k,o}$, we have
\[
    \underline{V}_h^{k,o}:=
    \bigoplus_{\begin{subarray}{c}
            K\in\Th\\i=1,\dots,n_K^{k,1}
    \end{subarray}}\mathrm{span}\{\phi_{K}^{i}\}
    \;\;\oplus
    \bigoplus_{\begin{subarray}{c}
            K\in\Th\\i=1,\dots,n_K^{k+1,2}
    \end{subarray}}\mathrm{span}\{\psi_K^i\},
\]
and
\begin{align*}
    \nabla\cdot\underline{v}_h^{o}|_K \in \poly^k(K)\backslash\mathbb{R},
    \quad (\underline{v}_h^{o}\cdot\underline{n}_K)|_F = 0, 
    \quad \forall F\in\partial K, K\in\Th, \quad \forall \underline{v}_h^{o}\in\underline{V}_h^{k,o},
\end{align*}
where $\phi_K^i$, $i=1,\dots,n_K^{k,1}$, is the higher order basis of the divergence-free bubbles on the element $K$ with zero normal components on $\Eh$; and $\psi_K^{i}$, $i=1,\dots n_K^{k+1,2}$, is the higher order basis on the element $K$ with zero normal components on $\Eh$ and nonzero divergence. 
Here the integers $n_F^k,n_K^{k,1},n_K^{k,2}$ are denoted as the number of basis functions of each corresponding group per facet/element. 
We also split $W_h^k$ into element-wise constant space and its complement, i.e.
\[
    W_h^k =
    {W}_h^{\partial} \oplus W_h^{k,o},\\
\]
where
\begin{align*}
    {W}_h^{\partial} &:= W_h^0,
    \\
    W_h^{k,o}&:=
    \left\{q\in W_h^k:\;\;(q,\;1)_K=0,\;\forall K\in\Th\right\}.
\end{align*}
Thus by the definition of the subspaces, we have
\begin{align*}
	\nabla\cdot\underline{V}_h^{k,\partial} &= {W}_h^{\partial},\\
	\nabla\cdot\underline{V}_h^{k,o} &= W_h^{k,o}.
\end{align*}
When $k=0$, there is no high-order local components in $\underline{V}_h^0$ and $W_h^0$. We refer readers to\cite[Section 5.2.7]{zag06} for more details of the basis functions of RT space.  

For two positive constants $a$ and $b$, we denote $a\lesssim b$ if there exists a positive constant $C$ independent of mesh size and model parameters 
such that $a\le Cb$. We denote $a \simeq b$ when $a\lesssim b$ and $b\lesssim a$.

\section{Generalized Stokes equation}
\label{sec:generalStokes}
\subsection{Model problem}
\label{subsec:stModelScheme}
Let $\underline{f}\in L^2(\Omega)$ be the source term and assume homogeneous Dirichlet boundary condition for simplicity. The model problem is to find $(\underline{u}, p)$ satisfying 
\begin{subequations}
    \label{stModel}
    \begin{alignat}{2}
        - \nabla\cdot(\nu\nabla\underline{u}) + \beta \underline{u} + \nabla p 
        =& \;\underline{f}, \quad &&\text{in $\Omega$,}
        \\
        \nabla\cdot\underline{u} =& \; 0, \quad&&\text{in $\Omega$,}
        \\
        \underline{u}=& \;\underline{0}, \quad &&\text{on $\partial\Omega$,}
    \end{alignat}
\end{subequations}
where $\underline{u}$ is the velocity, $p$ is the pressure, $\nu > 0$ is a constant representing the fluid viscosity, and the lower-order term coefficient $0\le\beta_0\le \beta\le \beta_1$.

To present the $H$(div)-HDG scheme for the generalized Stokes equation, we introduce the tensor $\dunderline{L}:= - \nu\nabla\underline{u}$ as a new variable and rewrite \eqref{stModel} into a first-order system:
\begin{subequations}
    \label{stModelMixed}
    \begin{alignat}{2}
        \nu^{-1}\dunderline{L} + \nabla\underline{u} =& \; 0, \quad&&\text{in $\Omega$,}
        \\
        \nabla\cdot\dunderline{L} + \beta\underline{u} 
        + \nabla p =& \; \underline{f}, \quad&&\text{in $\Omega$,}
        \\
        \nabla\cdot\underline{u} =& \; 0, \quad&&\text{in $\Omega$,}
        \\
        \underline{u} =& \; \underline{0}, \quad&&\text{on $\partial\Omega$.}
    \end{alignat}
\end{subequations}
Both the superconvergence property of the $H$(div)-HDG scheme used in this study and the geometric multigrid analysis for the lowest-order case require the following full elliptic regularity result for the solution $(\dunderline{L}, \underline{u}, p) \in \dunderline{H}^1(\Omega)\times(\underline{H}^2(\Omega)\cap\underline{H}^1_0(\Omega)\times (H^1(\Omega)\backslash\mathbb{R})$ to the model problem \eqref{stModelMixed}:
\begin{align}
    \label{stReg}
    \|\dunderline{L}\|_1 + \|\underline{u}\|_2 + \|p\|_1 
    \lesssim \|f\|_0,
\end{align}
which holds in convex domain $\Omega$.

\subsection{The $H$(div)-HDG scheme}
The $H$(div)-HDG scheme used in this study for the system \eqref{stModelMixed} is to find $(\dunderline{L}_h, \underline{u}_h, \widehat{\underline{u}}_h, p_h) \in\dunderline{W}_h^k \times\underline{V}_{h,0}^k \times\widehat{\underline{V}}_{h,0}^{k}\times W_{h,0}^{k}$, $k \geq 0$, such that

\begin{subequations}
    \label{stWeak}
    \begin{align}
    \label{stWeak1}
        (\nu^{-1}\dunderline{L}_h,\; \dunderline{G}_h)_{\mathcal{T}_h}  
        + (\nabla\underline{u}_h,\; \dunderline{G}_h)_{\mathcal{T}_h}
        - \langle \mathsf{tng}(\underline{u}_h - \widehat{\underline{u}}_h),\; \dunderline{G}_h\underline{n}_K\rangle_{\partial\mathcal{T}_h} &= 0,
        \\
    \label{stWeak2}
        -( \dunderline{L}_h,\; \nabla\underline{v}_h )_{\mathcal{T}_h}
        + \langle \dunderline{L}_h\underline{n}_K,\; \mathsf{tng}(\underline{v}_h - \widehat{\underline{v}}_h) \rangle_{\partial\Th}
        + (\beta\underline{u}_h,\; \underline{v}_h)_{\mathcal{T}_h} 
        - (p,\; \nabla\cdot\underline{v}_h)_{\Th}
        & = (\underline{f},\; \underline{v}_h)_{\mathcal{T}_h},
        \\
    \label{stWeak3}
        (\nabla\cdot\underline{u}_h,\; q_h)_{\mathcal{T}_h}
        &= 0,
    \end{align}
\end{subequations}
for all $(\dunderline{G}_h, \underline{v}_h, \widehat{\underline{v}}_h, q_h)\in\dunderline{W}_h^k \times\underline{V}_{h,0}^k \times\widehat{\underline{V}}_{h,0}^{k}\times W_{h,0}^{k}$.
The above $H$(div)-HDG scheme has been studied in \cite{fuQiuBrinkman} for the Brinkman equations. Besides the superconvergence property for post-processed velocity when $k\ge 1$, the numerical solution $(\dunderline{L}_h, \underline{u}_h, \widehat{\underline{u}}_h, p_h)$ to the $H$(div)-HDG scheme \eqref{stWeak} has optimal a priori error analysis results when $k\ge 0$ that are parameter-robust with respect to the ratio $\nu / \beta$. Since $\nabla\cdot\underline{V}_{h}^{k}=W_h^k$, the velocity error estimates are pressure-robust. We refer readers to \cite[Section 2]{fuQiuBrinkman} for more details.
Note that no extra facet-based HDG stabilization is introduced in the scheme \eqref{stWeak}, hence, it is technically a hybrid-mixed method. Here we follow the convention in  \cite{cockburnHdivHDGStokes,fuQiuBrinkman}, and still call it an HDG method.
Based on the subspace splitting in Section \ref{sec:notate}, we split the numerical solution $(\dunderline{L}_h, \underline{u}_h, \widehat{\underline{u}}_h, p_h)$ to \eqref{stWeak} into \textit{local} variables $(\dunderline{L}_h, \underline{u}_h^o, p_h^o)\in\dunderline{W}_h^k \times\underline{V}_{h}^{k,o}\times W_{h}^{k,o}$ and \textit{global} variables $(\underline{u}_h^{\partial}, \widehat{\underline{u}}_h, p_h^{\partial})\in\underline{V}_{h,0}^{k,\partial}\times\widehat{\underline{V}}_{h,0}^{k}\times W_{h,0}^{\partial}$, where
\[
    \underline{u}_h = \underline{u}_h^o + \underline{u}_h^{\partial},
    \quad p_h = p_h^o + p_h^{\partial}.
\]
When implementing the $H$(div)-HDG scheme \eqref{stWeak}, we first eliminate the local variables to arrive at the condensed global system composed of the global variables. After solving $(\underline{u}_h^{\partial}, \widehat{\underline{u}}_h, p_h^{\partial})$ from the condensed global system, which is the most computationally costly part, we recover the local variables in an element-by-element manner.
To simplify notation, we denote $\underline{\mathbb{v}}_h:=(\underline{v}_h, \widehat{\underline{v}}_h)$, $\underline{\mathbb{v}}_h^{\partial}:=(\underline{v}_h^{\partial}, \widehat{\underline{v}}_h)$,
$\underline{\mathbb{V}}_h^k:=\underline{V}_h^k\times \widehat{\underline{V}}_h^k$, 
and $\underline{\mathbb{V}}_h^{k,\partial}:=\underline{V}_h^{k,\partial}\times \widehat{\underline{V}}_h^k$ from now on.
We introduce an $L^2$-like inner-product on $\underline{\mathbb{V}}_{h}^{k,\partial}$:
\begin{align}
    \label{globalL2}
    (\underline{\mathbb{u}}_h^\partial,\; \underline{\mathbb{v}}_h^\partial)_{0, h}
    := 
        \sum_{K\in\Th}\frac{|K|}{d+1}
        \left(\langle \mathsf{nrm}({\underline{u}}_h),\; \mathsf{nrm}({\underline{v}}_h) \rangle_{\partial\Th}
        +
        \langle \mathsf{tng}(\widehat{\underline{u}}_h),\; \mathsf{tng}(\widehat{\underline{v}}_h) \rangle_{\partial\Th}\right)
\end{align}
for all $\underline{\mathbb{u}}_h^\partial,\underline{\mathbb{v}}_h^\partial \in \underline{\mathbb{V}}_h^{k, \partial}$, with the induced norm $\|\cdot\|_{0, h}$. The constant global pressure space $W_{h}^{\partial}$ is equipped with standard $L^2$ norm denoted as
\[
    [p_h^\partial,\; q_h^\partial]_{0,h} :=
    (p_h^\partial,\; q_h^\partial)_{\Th},
    \quad \forall p_h^\partial, q_h^\partial \in W_{h}^\partial.
\]
The characterization of the condensed system of \eqref{stWeak} has been studied in \cite{fuQiuBrinkman} and the operator form is to find $(\underline{\mathbb{u}}_h^\partial, p_h^{\partial})\in\underline{\mathbb{V}}_{h,0}^{k,\partial}\times W_{h,0}^{\partial}$ satisfying
\begin{subequations}
    \label{stOptForm}
    \begin{align}
        \underline{A}_{k,h}\underline{\mathbb{u}}_h^\partial
        + \underline{B}_{k,h}^\ast p_h^\partial =\;& \underline{F}_{k,h}
        \\
        \underline{B}_{k,h} \underline{\mathbb{u}}_h^\partial =\;& 0,
\end{align}
\end{subequations}
 for all $\underline{\mathbb{u}}_h^\partial, \underline{\mathbb{v}}_h^\partial\in\underline{\mathbb{V}}_{h,0}^{k,\partial}$ and $q_h^\partial \in W_{h,0}^\partial$, where
\begin{align*}
    (\underline{A}_{k,h}\underline{\mathbb{u}}_h^\partial,\; \underline{\mathbb{v}}_h^\partial)_{0,h}:=\;&
    (\nu\dunderline{\mathcal{L}}(\underline{\mathbb{u}}_h^\partial),\; \dunderline{\mathcal{L}}(\underline{\mathbb{v}}_h^\partial))_{\Th}
    + (\beta\underline{u}_h^\partial,\; \underline{v}_h^\partial)_{\Th},
    \\
    [\underline{B}_{k,h}\underline{\mathbb{u}}_h^\partial,\; q_h^\partial]_{0,h}:=\;&
    -(\nabla\cdot\underline{{u}}_h^\partial,\; q_h^\partial)_{\Th},
\end{align*}
and $\dunderline{\mathcal{L}}:\underline{\mathbb{V}}_{h,0}^{k,\partial}\rightarrow\dunderline{W}_{h}^{k}$ is a mapping defined by the local solvers of the $H$(div)-HDG scheme, and $\underline{B}_{k,h}^\ast:W_{h,0}^{\partial}\rightarrow\underline{\mathbb{V}}_{h,0}^{k,\partial}$ is the transpose of $\underline{B}_{k,h}:\underline{\mathbb{V}}_{h,0}^{k,\partial}\rightarrow W_{h,0}^{\partial}$ with respect to $L^2$ inner product:
\begin{align*}
    (\underline{B}_{k,h}^\ast p_h^\partial,\; \underline{\mathbb{v}}_h^\partial)_{0,h}
    = [p_h^\partial,\; \underline{B}_{k,h}\underline{\mathbb{v}}_h^\partial]_{0,h},
    \quad \forall p_h^\partial\in W_{h,0}^\partial,\; \underline{\mathbb{v}}_h^\partial\in\underline{\mathbb{V}}_{h,0}^{k,\partial}
\end{align*}
It is clear to see that $\underline{A}_{k,h}:\underline{\mathbb{V}}_{h,0}^{k,\partial}\rightarrow\underline{\mathbb{V}}_{h,0}^{k,\partial}$ is a symmetric positive definite (SPD) operator, and we denote $\|\cdot\|_{{A}_{k,h}}$ as the induced norm on $\underline{\mathbb{V}}_{h,0}^{k,\partial}$, i.e. $\|\cdot\|_{{A}_{k,h}} := \sqrt{(\underline{A}_{k,h}\cdot,\;\cdot)_{0,h}}$ 

\subsection{Augmented Lagrangian Uzawa iteration for the condensed $H$(div)-HDG}
To avoid solving saddle-point system, we apply the augmented Lagrangian Uzawa iteration method \cite{uzawa1958iterative, fortin2000augmented, lee2007robust, hong2016robust}.
. 
The saddle-point system is first transformed into an equivalent augmented Lagrangian formulation: find $(\underline{\mathbb{u}}_h^\partial, p_h^{\partial})\in\underline{\mathbb{V}}_{h,0}^{k,\partial}\times W_{h,0}^{\partial}$ satisfying
\begin{subequations}
    \label{stALform}
    \begin{align}
        \underbrace{(\underline{A}_{k,h}+ \epsilon^{-1}\underline{B}_{k,h}^\ast\underline{B}_{k,h})}_{\underline{A}_{k,h}^\epsilon}\underline{\mathbb{u}}_h^\partial 
        + \underline{B}_{k,h}^\ast p_h^\partial =& \underline{F}_{k, h},
        \\
        \underline{B}_{k,h} \underline{\mathbb{u}}_h^\partial =& 0,
    \end{align}
\end{subequations}
where $\epsilon$ is a small augmented Lagrangian parameter, which is also referred to as the penalty parameter. The operator form of $\epsilon^{-1}\underline{B}_{k,h}^\ast\underline{B}_{k,h}$ is expressed as
\[
    (\epsilon^{-1}\underline{B}_{k,h}^\ast\underline{B}_{k,h}\underline{\mathbb{u}}_h^\partial,\; \underline{\mathbb{v}}_h^\partial)_{0,h}
    = \epsilon^{-1}(\nabla\cdot\underline{{u}}_h^\partial,\;\nabla\cdot\underline{{v}}_h^\partial)_{\Th},
    \quad \forall \underline{\mathbb{u}}_h^\partial, \underline{\mathbb{v}}_h^\partial\in\underline{\mathbb{V}}_{h,0}^{k,\partial}.
\]
The Uzawa iteration method with $\epsilon^{-1} \gg 1$ is to start with $p_h^{\partial\,(0)}=0$, and iteratively find 
$(\underline{\mathbb{u}}_h^{\partial\,(n)},p_h^{\partial\,(n)})\in\underline{\mathbb{V}}_{h,0}^{k,\partial}\times W_{h,0}^{\partial}$ such that 
\begin{subequations}
    \label{stOpEq-ALU}
    \begin{align}
    \label{stOpEq-ALU1}
    \underline{A}_{k,h}^\epsilon \underline{\mathbb{u}}_h^{\partial\,(n)} =& \;\underline{F}_{k,h} - \underline{B}_{k,h}^\ast p_h^{\partial\,(n-1)},
        \\
    \label{stOpEq-ALU2}
    p_h^{\partial\, (n)} =& \;   p_h^{\partial\, (n-1)}- \epsilon^{-1}  \underline{B}_{k,h} \underline{\mathbb{u}}_h^{\partial\, (n)},
\end{align}
\end{subequations}
The convergence property of the Augmented Lagrangian Uzawa iteration method \eqref{stOpEq-ALU} has been studied in \cite{lee2007robust}, and we quote it here for completeness.
\begin{lemma}[Lemma 2.1 of \cite{lee2007robust}]
    \label{lem:uzawaConverge}
    Let $(\underline{\mathbb{u}}_h^{\partial}, p_h^\partial)\in\underline{\mathbb{V}}_{h,0}^{k,\partial}\times W_{h,0}^\partial$ be the solution of \eqref{stOptForm} and 
    $(\underline{\mathbb{u}}_h^{\partial\, (n)}, p_h^{\partial\, (n)})$
    be the $n$-th Uzawa iteration solution to \eqref{stOpEq-ALU}. 
Then the following estimate holds:
    \begin{align*}
        \|\underline{\mathbb{u}}_h^{\partial\,(n)} - \underline{\mathbb{u}}_h^{\partial}\|_{{A}_{k,h}}
        \lesssim&
        \sqrt{\epsilon}\|p_h^{\partial\,(n)} - p_h^\partial\|_0\;
        \lesssim\;
        \sqrt{\epsilon}\bigl(\frac{\epsilon}{\epsilon+\mu_0}\bigr)^{n}
        \|p_h^\partial\|_{0},
    \end{align*}
    where $\mu_0$ is the 
    minimal eigenvalue of the Schur complement
    $\underline{S}_{k,h}=\underline{B}_{k,h}\underline A_{k,h}^{-1}\underline{B}^\ast_{k,h}$.
\end{lemma}

The discontinuous piecewise constant nature of the global pressure space $W_{h,0}^\partial$ results in trivial computation in equation \eqref{stOpEq-ALU2}. As a consequence, the major computational cost of a Uzawa iteration is associated with solving the global velocity equation \eqref{stOpEq-ALU1}. Here we use the  preconditioned conjugate gradient (PCG) method to iteratively solve it. Specifically, we refer to the augmented Uzawa iteration method as the \textit{outer} iteration, and the CG method for solving equation \eqref{stOpEq-ALU1} as the \textit{inner} iteration. The subsequent subsections are devoted to designing an $hp$-multigrid algorithm to precondition the operator $\underline{A}_{k,h}^\epsilon$ for the PCG method.

\begin{remark}[On the augmented Lagrangian Uzawa iteration for \eqref{stOptForm}]
\label{remark:ALUZ}
    We are aware of the extensive body of literature on solving saddle-point problems, including the problem presented in equation \eqref{stOptForm}. For a comprehensive review, we refer the reader to \cite{benziSaddle2005}. In this work, we adopt the augmented Lagrangian Uzawa iteration method for two primary reasons. Firstly, by locally eliminating $p_h^\partial$, the final matrix size is further reduced. Secondly, the resulting global velocity operator $\underline{A}_{k,h}^\epsilon$ is SPD on the space $\underline{\mathbb{V}}_{h,0}^{k,\partial}$, allowing us to use the CG method, for which the eigenvalues of the preconditioned linear operator are sufficient to characterize the convergence rate of the method. A good preconditioner for $\underline{A}_{k,h}^\epsilon$ can also serve in block preconditioners as an approximate inverse of the primal $(1,1)$-block of the augmented Lagrangian transformed saddle-point structure as in \eqref{stALform}.

    However, there are two potential drawbacks to this approach. Firstly, the term $\epsilon^{-1}\underline{B}_{k,h}^\ast\underline{B}_{k,h}$ in $\underline{A}_{k,h}^\epsilon$ poses challenges for preconditioning when $\epsilon^{-1}\gg 1$. Secondly, setting $\epsilon$ to an extremely small value leads to round-off issues. This is because the non-zero entries in the matrix $\underline A_{k,h}$ are of order $\mathcal{O}(1)$, whereas the entries in $\epsilon^{-1}\underline{B}_{k,h}^\ast \underline{B}_{k,h}$ are of order $\mathcal{O}(\epsilon^{-1})$. 
    Hence, $\log(\epsilon^{-1})$ digits loss is expected in a practical implementation.
    In our numerical experiments, we 
    use double digit calculation with machine precision of $10^{-16}$,  set $\epsilon=10^{-6}$, and perform two Uzawa iterations.
Here the round-off error is of order 
$\mathcal{O}(10^{-16}/\epsilon)=\mathcal{O}(10^{-10})$.
\end{remark}

\subsection{Geometric $h$-multigrid for the lowest-order scheme}
We first focus on the lowest order case when $k=0$ and prove the equivalence between the condensed $H$(div)-HDG scheme \eqref{stWeak} and a CR discretization after a pressure-robust treatment. Then we propose for the lowest-order scheme a geometric $h$-multigrid method robust with the mesh size and the penalty parameter $\epsilon$.

\subsubsection{Equivalence to a CR discretization}
To explain the pressure-robust treatment of the CR discretization, we introduce an interpolation operator from vector CR space to the RT0 space $\underline{\Pi}^{RT}:\underline{V}_h^{\CR}\rightarrow \underline{V}_h^0$ satisfying 
\begin{align}
    \label{CR_2RT0}
    \int_F (\underline{\Pi}^{RT}\underline{v}_h^{CR})\cdot \underline{n}_K \mathrm{ds} = \int_F \underline{v}_h^{CR}\underline{n}_K \mathrm{ds},
    \quad \forall F\in\partial K, K\in\Th,\;
    \forall \underline{v}_h^{CR}\in\underline{V}_h^{\CR}.
\end{align} 
To establish the link between the lowest-order $H$(div)-HDG and the CR discretization, we define an interpolation operator from the lowest order RT0 and tangential facet finite element space to the vector CR space $\underline{\Pi}^{\CR}: \underline{\mathbb{V}}_h^0 \rightarrow \underline{V}_h^{\CR}$:
\begin{subequations}
\label{RT0N1_2CR}
\begin{align}
    \underline{\Pi}^{\CR}\underline{\mathbb{v}}_h (m_F) \cdot\underline{n}_K 
    & = \underline{v}_h(m_F)\cdot\underline{n}_K,
    \\
    \mathsf{tng}(\underline{\Pi}^{\CR}\underline{\mathbb{v}}_h(m_F))
    & = \mathsf{tng}(\widehat{\underline{v}}_h(m_F)),
\end{align}
\end{subequations}
for all $F \in \partial K$, $K\in\Th$, where $m_F$ is the barycenter of facet $F$.
By counting the DOFs of the spaces, it is clear that $\underline{\Pi}^{\CR}$ is an isomorphic mapping between $\underline{\mathbb{V}}_h^0$ and $\underline{V}_h^{\CR}$. Then we have the following property for the above interpolation operators:
\begin{lemma} 
\label{lem:invarInterp}
For all $\underline{\mathbb{v}}_h\in\underline{\mathbb{V}}_h^0$, we have:
\begin{align}
    \label{invarInterp}
    \underline{\Pi}^{RT}\underline{\Pi}^{CR}\underline{\mathbb{v}}_h &= \underline{v}_h,
    \\
    \label{invarDiv}
    \nabla\cdot\underline{\Pi}^{CR}\underline{\mathbb{v}}_h &= 
    \nabla\cdot\underline{v}_h.
\end{align}
\end{lemma}
\begin{proof}
    With the definition of the interpolation operators in \eqref{CR_2RT0}-\eqref{RT0N1_2CR},
    we have
    \begin{align*}
        \int_F (\underline{\Pi}^{RT}\underline{\Pi}^{CR}\underline{\mathbb{v}}_h)\cdot \underline{n}_K \mathrm{ds} 
        = 
        \int_F (\underline{\Pi}^{CR}\underline{\mathbb{v}}_h)\cdot\underline{n}_K \mathrm{ds}
        = &
        \int_F \underline{v}_h\cdot\underline{n}_K \mathrm{ds},
    \end{align*}
    for all $\underline{\mathbb{v}}_h\in\underline{\mathbb{V}}_h^0$, $F\in\partial K$ and $K\in\Th$.
    The result \eqref{invarInterp} then immediately follows. 
    
    Similarly, by the divergence theorem, we have for all $\underline{\mathbb{v}}_h\in\underline{\mathbb{V}}_h^0$,
    \begin{align*}
        \int_{K}\nabla\cdot\underline{\Pi}^{CR}\underline{\mathbb{v}}_h \mathrm{dx}
        =
        \int_{\partial K}(\underline{\Pi}^{CR}\underline{\mathbb{v}}_h)\cdot\underline{n}_K\mathrm{ds}
        =
        \int_{\partial K}\underline{{v}}_h\cdot\underline{n}_K \mathrm{ds}
        =
        \int_{K} \nabla\cdot\underline{v}_h\mathrm{dx},
    \end{align*}
    and the divergence invariance \eqref{invarDiv} follows from the fact that 
    $$(\nabla\cdot\underline{\Pi}^{CR}\underline{\mathbb{v}}_h)|_K,
    \quad (\nabla\cdot\underline{v}_h)|_K \in \poly^0(K), 
    \quad\forall K\in\Th.$$
\end{proof}
Now we are ready to present our main result below.

\begin{theorem}[Equivalence to CR discretization at lowest order]
\label{theo:stEquivCR}
Let $(\underline{u}_{h}^{CR}, p_h^{CR}) \in \underline{V}_{h,0}^{CR}\times W_{h,0}^0$ be the solution to the following nonconforming scheme:
\begin{subequations}
    \label{equivCRstWeak}
    \begin{align}
        \nu(\nabla\underline{u}_h^{CR},\; \nabla\underline{v}_h^{CR})_{\Th}
        +\beta(\underline{\Pi}^{RT}\underline{u}_h^{CR},\; \underline{\Pi}^{RT}\underline{v}_h^{CR})_{\Th}
        - (p_h^{CR},\; \nabla\cdot\underline{v}_h^{CR})_{\Th} 
        &= (\underline{f},\; \underline{\Pi}^{RT}\underline{v}_h^{CR})_{\Th},
        \\
        (\nabla\cdot\underline{u}_h^{CR},\; q_h^{CR})_{\Th}
        &= 0,
    \end{align}
\end{subequations}
for all $(\underline{v}_{h}^{CR}, q_h^{CR}) \in \underline{V}_{h,0}^{CR}\times W_{h,0}^0$. Then the solution $(\dunderline{L}_h, \mathbb{\underline{u}}_h, p_h) \in \dunderline{W}_h^0\times\underline{\mathbb{V}}_{h,0}^0 \times W_{h,0}^0$ to the lowest-order $H$(div)-HDG scheme for the generalized Stokes equation \eqref{stWeak} satisfies
\begin{subequations}
\label{solRelation}
    \begin{align}
    \label{LRelation}
    \dunderline{L}_h & = -\nu\nabla\underline{u}_h^{CR},
    \\
    \label{uRelation}
    \underline{\Pi}^{CR}\underline{\mathbb{u}}_h
    & = \underline{u}_h^{CR},
    \\
    \label{pRelation}
    p_h &= p_h^{CR}.
\end{align}
\end{subequations}
\end{theorem}
\begin{proof}
By integration by parts and the definition of $\underline{\Pi}^{CR}$, we have
\begin{align}
    \label{equivProf1}
    (\nabla\underline{u}_h,\; \dunderline{G}_h)_{\mathcal{T}_h}
    - \langle \mathsf{tng}(\underline{u}_h - \widehat{\underline{u}}_h),\; \dunderline{G}_h\underline{n}_K\rangle_{\partial\mathcal{T}_h}
    = & 
    \langle (\mathsf{nrm}(\underline{u}_h) + \mathsf{tng}(\widehat{\underline{u}}_h),\; \dunderline{G}_h\underline{n}_K\rangle_{\partial\mathcal{T}_h}\\
    &\;\nonumber
    - \underbrace{(\underline{u}_h,\; \nabla\cdot\dunderline{G}_h)_{\mathcal{T}_h}}_{\equiv 0}
    \\
    \nonumber
    = &
    \langle \underline{\Pi}^{CR}\underline{\mathbb{u}}_h,\; \dunderline{G}_h\underline{n}_K\rangle_{\partial\mathcal{T}_h}
    \\
    \nonumber
    = &
    (\nabla\underline{\Pi}^{CR}\underline{\mathbb{u}}_h,\; \dunderline{G}_h)_{\mathcal{T}_h}
    + \underbrace{( \underline{\Pi}^{CR}\underline{\mathbb{u}}_h,\; \nabla\cdot\dunderline{G}_h)_{\mathcal{T}_h}}_{\equiv 0},
\end{align}
for all $\underline{\mathbb{u}}_h\in\underline{\mathbb{V}}_{h,0}^0$ 
and $\dunderline{G}_h\in\dunderline{W}_{h}^0$,
where we used the fact that $\dunderline{G}_h|_K\in\dunderline\poly^0(K)$. 
By plugging \eqref{equivProf1} into \eqref{stWeak1}, we immediately get 
\begin{align}
\label{equivProf2}
    \dunderline{L}_h 
    = -\nu\nabla\underline{\Pi}^{CR}\underline{\mathbb{u}}_h.
\end{align}

Next, with the same arguments as in proving \eqref{equivProf1}, for all $\underline{\mathbb{v}}_h\in\underline{\mathbb{V}}_{h,0}^{0}$ and $\dunderline{L}_h\in\dunderline{W}_{h}^0$ we have
\begin{align}
    \label{equivProf3}
    -( \dunderline{L}_h,\; \nabla\underline{v}_h )_{\mathcal{T}_h}
    + \langle \dunderline{L}_h\underline{n}_K, \mathsf{tng}(\underline{v}_h - \widehat{\underline{v}}_h) \rangle_{\partial\Th}
    =&
    \underbrace{(\nabla\cdot\dunderline{L}_h,\; \underline{v}_h)_{\Th}}_{\equiv 0} 
    - \langle \dunderline{L}_h\underline{n}_K,\; \mathsf{nrm}(\underline{u}_h) + \mathsf{tng}(\widehat{\underline{v}}_h)
    \rangle_{\partial\Th}
    \\
    \nonumber
    =&
    -\langle\dunderline{L}_h\underline{n}_K,\;
    \underline{\Pi}^{CR}\underline{\mathbb{v}}_h\rangle_{\partial\Th}
    \\
    \nonumber
    =&
    \underbrace{-(\nabla\cdot\dunderline{L}_h,\; \Pi^{CR}\underline{\mathbb{v}}_h)_{\Th}}_{\equiv 0} 
    - (\dunderline{L}_h,\; \nabla\underline{\Pi}^{CR}\underline{\mathbb{v}}_h)_{\Th}
    \\
    \nonumber
    =&
    \nu(\nabla\underline{\Pi}^{CR}\underline{\mathbb{u}}_h,\; \nabla\underline{\Pi}^{CR}\underline{\mathbb{v}}_h)_{\Th},
\end{align}
where we plug in the equivalence \eqref{equivProf2} at the final step. Finally, by plugging \eqref{equivProf3} into the lowest order $H$(div)-HDG scheme \eqref{stWeak2} and results in Lemma \ref{lem:invarInterp}, the lowest order $H$(div)-HDG scheme \eqref{stWeak} becomes finding $(\underline{\mathbb{u}}_h,\; p_h)\in\underline{\mathbb{V}}_{h,0}^0\times W_{h,0}^0$ satisfying
\begin{subequations}
    \label{hdivHdg-p0}
    \begin{align}
     \nu(\nabla\underline{\Pi}^{CR}\underline{\mathbb{u}}_h,\; \nabla\underline{\Pi}^{CR}\underline{\mathbb{v}}_h)_{\Th}
     +\beta(\underline{\Pi}^{RT}\underline{\Pi}^{CR}\underline{\mathbb{u}}_h,\; \underline{\Pi}^{RT}\underline{\Pi}^{CR}\underline{\mathbb{v}}_h)_{\mathcal{T}_h} &
     \\ \nonumber
     -(p_h,\; \nabla\cdot\underline{\Pi}^{CR}\underline{\mathbb{v}}_h)_{\Th} 
     & = (\underline{f},\; \underline{\Pi}^{RT}\underline{\Pi}^{CR}\underline{\mathbb{v}}_h)_{\mathcal{T}_h}
     \\
     (\nabla\cdot\underline{\Pi}^{CR}\underline{\mathbb{u}}_h,\; q_h)_{\mathcal{T}_h}
     &= 0,
\end{align}
\end{subequations}
for all $(\underline{\mathbb{v}}_h,\; q_h)\in\underline{\mathbb{V}}_{h,0}^0\times W_{h,0}^0$, and the equivalence results \eqref{solRelation} follows from the well-posedness of the CR discretization \eqref{equivCRstWeak}.
\end{proof} 

\begin{remark}[On the modified CR discretization]
\label{rem:modifyCR}
The modified CR discretization \eqref{equivCRstWeak} was firstly introduced in \cite{linkeRT2014} for the incompressible Stokes equation, i.e. $\beta=0$ in \eqref{stModel}. The original mixed CR discretization for the incompressible Stokes equation, though inf-sup stable, has poor mass conservation property and not pressure-robust velocity error estimates. To address this issue, the test function on the right hand side was reconstructed using $\underline{\Pi}^{RT}$, which maps discretely divergence-free test functions onto exactly divergence-free test functions and reconstructs the $L^2$-orthogonality between discretely divergence-free and irrotational vector fields in the mixed method, while the left hand side remains unchanged. Despite introducing extra inconsistency error into the scheme, optimal and pressure-robust discrete $H^1$ norm error estimates of velocity were obtained, and optimal and pressure-robust $L^2$ norm convergence rates were supported by numerical experiments. Later this pressure-robust treatment of CR discretization was extended to the incompressible generalized Stokes equation \cite{linkNS2016}, i.e. $\beta\neq 0$ in \eqref{stModel}, where both the trial function and the test function in the mass term were further reconstructed onto RT0 space. It is worth mentioning that a skew-symmetric pressure-robust treatment for the convection term was also introduced in \cite{linkNS2016} for the CR discretization of the incompressible Navier-Stokes equations.
\end{remark}

In practical implementation of the lowest-order $H$(div)-HDG scheme \eqref{stWeak}, we first locally eliminate $\dunderline{L}_h$ and arrive at the condensed global system \eqref{hdivHdg-p0} composed of $(\underline{\mathbb{u}}_h,\;p_h)$ where there is no higher-order local velocity and pressure components. Thus theorem \ref{theo:stEquivCR} implies the equivalence between the condensed lowest-order $H$(div)-HDG scheme \eqref{stWeak} and the modified CR discretization \eqref{equivCRstWeak}. In other words, when $k=0$, the condensed $H$(div)-HDG scheme is equivalent to find $(\underline{\mathbb{u}}_h,p_h)\in\underline{\mathbb{V}}_{h,0}^0\times W_{h,0}^0$ satisfying
\begin{subequations}
\label{k0-stOptForm}
    \begin{align}
        \underline{A}_{0,h}\underline{\mathbb{u}}_h
        + \underline{B}_{0,h}^\ast p_h =& \underline{F}_{0,h},
        \\
        \underline{B}_{0,h} \underline{\mathbb{u}}_h =& 0,
    \end{align}
\end{subequations}
where for all $\underline{\mathbb{u}}_h, \underline{\mathbb{v}}_h \in\underline{\mathbb{V}}_{h,0}^{0}$ and $q_h\in W_{h,0}^0$,
\begin{align*}
    (\underline{A}_{0,h}\underline{\mathbb{u}}_h,\;\underline{\mathbb{v}}_h)_{0,h}
    \equiv &\;\;
    \nu(\nabla\underline{\Pi}^{CR}\underline{\mathbb{u}}_h,\; \nabla\underline{\Pi}^{CR}\underline{\mathbb{v}}_h)_{\Th}
     +\beta(\underline{\Pi}^{RT}\underline{\Pi}^{CR}\underline{\mathbb{u}}_h,\; \underline{\Pi}^{RT}\underline{\Pi}^{CR}\underline{\mathbb{v}}_h)_{\mathcal{T}_h},
     \\
     [\underline{B}_{0,h} \underline{\mathbb{u}}_h, q_h]_{0,h}
     \equiv &\;\;
      -(\nabla\cdot\underline{\Pi}^{CR}\underline{\mathbb{u}}_h,\; q_h)_{\mathcal{T}_h},
\end{align*}
and the augmented Lagrangian Uzawa iteration \eqref{stOpEq-ALU} is equivalent to find $(\underline{\mathbb{u}}_h^{(n)},p_h^{(n)})\in\underline{V}_{h,0}^0\times W_{h,0}^0$ satisfying
\begin{subequations}
\label{k0-stOptEq-ALU}
    \begin{align}
    \label{k0-stOptEq-ALU1}
    \underline{A}_{0,h}^\epsilon \underline{\mathbb{u}}_h^{(n)} =& \;\underline{F}_{0,h} - \underline{B}_{0,h}^\ast p_h^{(n-1)},
    \\
    \label{k0-stOptEq-ALU2}
    p_h^{(n)} =& \;   p_h^{(n-1)}- \epsilon^{-1}  \underline{B}_{0,h} \underline{\mathbb{u}}_h^{(n)}.
\end{align}
\end{subequations}
where for all $\underline{\mathbb{u}}_h, \underline{\mathbb{v}}_h \in\underline{\mathbb{V}}_{h,0}^{0}$,
\begin{align*}
    (\underline{A}_{0,h}^\epsilon \underline{\mathbb{u}}_h,\; \underline{\mathbb{v}}_h)_{0,h}
    \equiv&\;\;
    \nu(\nabla\underline{\Pi}^{CR}\underline{\mathbb{u}}_h,\; \nabla\underline{\Pi}^{CR}\underline{\mathbb{v}}_h)_{\Th}
     +\beta(\underline{\Pi}^{RT}\underline{\Pi}^{CR}\underline{\mathbb{u}}_h,\; \underline{\Pi}^{RT}\underline{\Pi}^{CR}\underline{\mathbb{v}}_h)_{\mathcal{T}_h} 
    \\&
    +
    (\nabla\cdot\underline{\Pi}^{CR}\underline{\mathbb{u}}_h,\; \nabla\cdot\underline{\Pi}^{CR}\underline{\mathbb{v}}_h)_{\mathcal{T}_h}
\end{align*}
This result is inspired by the equivalence of the lowest order Raviart-Thomas mixed method and the nonconforming CR method for Poisson's equation \cite{arnold1985mixed, marini1985inexpensive}. Such equivalence allows for the direct application of the rich multigrid theory on the CR discretization to solve the condensed lowest-order $H$(div)-HDG scheme.  In this study, we adopt Sch\"oberl's geometric multigrid theory \cite{schoberl1999multigrid, schoberl1999robust} to the SPD primal operator $\underline{A}_{0,h}^{\epsilon}$ to get geometric multigrid methods robust concerning both mesh size and the penalty parameter $\epsilon$.

\begin{remark}[On geometric $h$-multigrid for the CR discretization]
    There are other two approaches in the literature to construct geometric multigrid algorithms for the CR scheme \eqref{equivCRstWeak} which is equivalent to the condensed system \eqref{hdivHdg-p0}.
    The first approach \cite{brenner1990nonconforming, turek1994multigrid, stevenson1998cascade} exploits the cell-wise divergence-free subspace of the CR element and implements multigrid algorithms for the positive definite system on the divergence-free kernel space. However, this method is limited to two dimensions and the extension to three dimensions is exceedingly challenging due to the need for constructing complex intergrid transfer operators between the divergence-free subspaces. 
    The second approach, proposed by Brenner \cite{brenner1993nonconforming, brenner1994nonconforming}, directly deals with the saddle point system (with a penalty term) to avoid the divergence kernel, and proved convergence results for nested W-cycle multigrid method with large enough smoothing steps.
    Sch\"oberl's approach works with the resulting positive definite primal operator from the saddle point system with a penalty term, the same as the $\underline{A}_{0,h}^{\epsilon}$ in our study. This approach is more feasible in three dimensions as the intergrid transfer operator is considerably easier to implement in practice than the first approach. Originally introduced for the $P^2$-$P^0$ discretization on triangles, Sch\"oberl's approach has been successfully applied by other researchers to other finite element schemes in two- and three-dimensions, as documented in \cite{Lee09, hong2016robust, Kanschat15,farrell2019augmented, Farrell20}.
\end{remark}

\subsubsection{Geometric multigrid algorithm}
In this subsection, we present detailed geometric multigrid algorithm to solve \eqref{k0-stOptEq-ALU1} based on the parameter-robust multigrid theory of Sch\"oberl \cite{schoberl1999multigrid, schoberl1999robust}.

We consider a hierarchical mesh sequence for the geometric $h$-multigrid algorithm, beginning with the coarsest simplicial triangulation $\mathcal{T}_1$. 
The finest mesh is denoted as $\mathcal{T}_J=\mathcal{T}_h$ and is obtained through a sequence of mesh refinements for $l=2,\dots, J$. 
On each mesh level, the mesh skeleton is denoted as $\mathcal{E}_l$ and the maximum mesh size of $\mathcal{T}_l$ is denoted by $h_l$. We assume that on each level the triangulation $\mathcal{T}_l$ is conforming, shape-regular, and quasi-uniform over the domain $\Omega$, and that the difference in mesh size between two adjacent mesh levels is bounded by $h_{l} \lesssim h_{l+1}$. The corresponding finite element spaces on $\mathcal{T}_l$ are denoted as $W_l^0$, $\underline{V}_l^0$, $\widehat{\underline{V}}_h^0$, and $V_l^{CR}$.
The corresponding $L^2$ inner product on global spaces $\underline{\mathbb{V}}_{l}^{0}$ and ${W}_l^0$ are denoted as $(\cdot,\;\cdot)_{0,l}$ and $[\cdot,\;\cdot]_{0,l}$,
and the corresponding linear operators in \eqref{k0-stOptEq-ALU} are denoted as $\underline{A}_{0,l}^\epsilon$, $\underline{B}_{0,l}^\ast$ and $\underline{B}_{0,l}$.

The main components in \cite{schoberl1999multigrid,schoberl1999robust} are (i) a robust intergrid transfer operator that transfer coarse-grid divergence-free functions to fine-grid (nearly) divergence-free functions,  and (ii) a robust block-smoother capable of capturing the divergence-free basis functions.

For the intergrid transfer operator, we first define the following averaging operator $\underline{I}_{l-1}^l:\underline{\mathbb{V}}_{l-1,0}^{0}\rightarrow\underline{\mathbb{V}}_{l,0}^{0}$ in light of the one used in multigrid methods for the CR discretization for Poisson's equation
\cite{brenner1989optimal, braess1990multigrid}: Let $(\underline{v}_{l}', \widehat{\underline{v}}_{l}') = \underline{I}_{l-1}^l\underline{\mathbb{v}}_{l-1}$, then we have
\begin{subequations}
\label{avgOpt}
    \begin{align}
    \mathsf{nrm}\left(\underline{v}_l'(m_F)\right):=& \left\{
        \renewcommand{\arraystretch}{1.5} 
        \begin{array}{ll}
            \mathsf{nrm}\left(
            (\Pi^{\CR}_{l-1}\underline{\mathbb{v}}_{l-1})(m_F)\right),& \text{if $F\in \mathcal E_l^0$ lies inside $\mathcal{T}_{l-1}$},  \\[1ex]
            \mathsf{nrm}\left(\frac{1}{2}\left((\Pi^{\CR}_{l-1}\underline{\mathbb{v}}_{l-1})^{+}(m_F) + (\Pi_{l-1}^{\CR}\underline{\mathbb{v}}_{l-1})^{-}(m_F)\right)\right),& \text{if $F\in \mathcal E_l^0$ lies on $\mathcal{E}_{l-1}^o$},
        \end{array}
    \right.
    \\
    \mathsf{tng}\left(\widehat{\underline{v}}_l'(m_F)\right):=& \left\{
        \renewcommand{\arraystretch}{1.5} 
        \begin{array}{ll}
            \mathsf{tng}\left(
            (\Pi^{\CR}_{l-1}\underline{\mathbb{v}}_{l-1})(m_F)\right),& \text{if $F\in \mathcal E_l^0$ lies inside $\mathcal{T}_{l-1}$},  \\[1ex]
            \mathsf{tng}\left(\frac{1}{2}\left((\Pi^{\CR}_{l-1}\underline{\mathbb{v}}_{l-1})^{+}(m_F) + (\Pi_{l-1}^{\CR}\underline{\mathbb{v}}_{l-1})^{-}(m_F)\right)\right),& \text{if $F\in \mathcal E_l^0$ lies on $\mathcal{E}_{l-1}^o$},
        \end{array}
    \right.
\end{align}
\end{subequations}
where $(\Pi^{\CR}_{l-1}\underline{\mathbb{v}}_{l-1})^{+}$ and $(\Pi^{\CR}_{l-1}\underline{\mathbb{v}}_{l-1})^{-}$ are the values of $\Pi^{\CR}_{l-1}\underline{\mathbb{v}}_{l-1}$ on two adjacent elements $K^{+},\; K^{-}\in{\mathcal{T}_{l-1}}$ that share the facet $F$. On the finer $l$-th mesh level, we denote $\underline{V}_{l,0}^{0,T}$ and $\widehat{\underline{V}}_{l,0}^{0,T}$ as the local subspaces of $\underline{V}_{l,0}^{0}$ and $\widehat{\underline{V}}_{l,0}^{0}$ respectively, with DOFs of these subspaces vanishing on the mesh skeleton $\mathcal{E}_{l-1}$ of the coarser $(l-1)$-th mesh level. Due to the components of $\underline{\mathbb{v}}_l'$ in $\underline{\mathbb{V}}_{l,0}^{0,T}$, the energy norm $\|\underline{\mathbb{v}}_{l-1}\|_{{A}_{0,l-1}^\epsilon}$ can not be bounded by $\|\underline{\mathbb{v}}_{l}'\|_{{A}_{0,l}^\epsilon}$ independent of $\epsilon^{-1}$ after performing only averaging operator, thus we stabilize this averaging operator with a local correction using discrete harmonic extensions \cite{schoberl1999multigrid}.
The integer grid transfer operator $\underline{\mathcal{I}}_{l-1}^l : \underline{\mathbb{V}}_{l-1,0}^0\rightarrow\underline{\mathbb{V}}_{l,0}^0$
is defined as follows:
\begin{align}
\label{div-transfer}
    \underline{\mathcal{I}}_{l-1}^{l}:=(id - \underline{P}^{T}_{{A}_{0,l}^\epsilon})\underline{I}_{l-1}^l,
\end{align}
where $id$ is the identity operator and  the local projection
$\underline{P}_{{A}_{0,l}^\epsilon}^{T}: {\underline{V}}_{l,0}^0\rightarrow {\underline{\mathbb{V}}}_{l,0}^{0,T}$ satisfies
\begin{align}
\label{harmonic-ex}
    (\underline{A}_{0,l}^\epsilon \underline{P}_{{A}_{0,l}^{\epsilon}}^{T}{\underline{\mathbb{u}}}_l,\; {\underline{\mathbb{v}}}_l^T)_{0,l}
    =
    (\underline{A}_{0,l}^\epsilon {\underline{\mathbb{u}}}_l,\; {\underline{\mathbb{v}}}_l^T)_{0,l}, 
    \quad \forall {\underline{\mathbb{u}}}_l\in\underline{\mathbb{V}}_{l,0}^0,\; {\underline{\mathbb{v}}}_l^T\in\underline{\mathbb{V}}_{l,0}^{0,T},
\end{align}
which is locally solved on coarse mesh elements of $\mathcal{T}_{l-1}$. We further define the restriction operator $\underline{\mathcal{I}}_l^{l-1}: \underline{\mathbb{V}}_{l,0}^0\rightarrow \underline{\mathbb{V}}_{l-1,0}^0$ as the transpose of $\underline{I}_{l-1}^l$ with respect to $(\cdot,\;\cdot)_{0,l}$ satisfying
\[
    (\underline{\mathcal{I}}_{l}^{l-1}\underline{\mathbb{u}}_{l},\; \underline{\mathbb{v}}_{l-1})_{0,l-1}
    =
    (\underline{\mathbb{u}}_{l},\; \underline{\mathcal{I}}_{l-1}^l\underline{\mathbb{v}}_l)_{0,l}, 
    \quad \forall \underline{\mathbb{u}}_l\in\underline{\mathbb{V}}_{l,0}^0,\; \underline{\mathbb{v}}_{l-1}\in\underline{\mathbb{V}}_{l-1,0}^{0}.
    \]

For the smoother for \eqref{k0-stOptEq-ALU1}, we employ the classical block smoother for $H$(div)-elliptic problems proposed by Arnold, Falk, and Winther \cite{arnold2000multigrid} to address the discretely divergence-free kernel space of the operator $\underline{A}_{0,l}^\epsilon$. Specifically, we utilize the vertex-patched damped block Jacobi or block Gauss-Seidel smoother.
It is worth noting that in three dimensions, edge-patch block smoothers can also be utilized to reduce memory usage \cite{arnold2000multigrid}. For completeness, we present the formulation of the vertex-patch damped block Jacobi smoother here.
We define $\EuScript{S}_l$ as the set of vertices in the triangulation $\mathcal{T}_l$,
and we further denote $\mathcal{T}_{l}^s$ as the subset of mesh elements
and $\mathcal{E}_{l}^s$ as the subset of mesh skeletons
meeting at the vertex
$s\in \EuScript{S}_l$ respectively, i.e.
\begin{align*}
    \mathcal{T}_{l}^s:=& \bigcup\limits_{\substack{K\in\mathcal{T}_l, \; s\in K}} K,
    \\
    \mathcal{E}_{l}^s:=& \bigcup\limits_{\substack{F\in\mathcal{E}_l, \; s\in F}} F.
\end{align*}
The lowest-order compound finite element space $\underline{\mathbb{V}}_{l,0}^0$ is then decomposed into overlapping subspaces with support on $\mathcal{T}_{l}^s$ and $\mathcal{E}_{l}^s$ as follows: 
\begin{align*}
    \underline{\mathbb{V}}_{l,0}^0 = \sum_{s\in \EuScript{S}_l} \underline{\mathbb{V}}_{l, 0}^{0, s} :=
    \sum_{s\in \EuScript{S}_l}
    \left\{
        (\underline{v}_l^s, \widehat{\underline{v}}_l^s)\in\underline{\mathbb{V}}_{l,0}^0:\;
        \mathrm{supp}\;{\underline{v}}_l^s\subset\mathrm{interior}(\mathcal{T}_{l}^s),\;
        \mathrm{supp}\;\widehat{\underline{v}}_l^s\subset\mathcal{E}_{l}^s
    \right\}.
\end{align*}
Furthermore, we define $\underline{P}_{{A}_{0,l}^\epsilon}^s:\; \underline{\mathbb{V}}_{l,0}^0\rightarrow\underline{\mathbb{V}}_{l,0}^{0,s}$ as the local projection onto the subspace $\underline{\mathbb{V}}_{l,0}^{0,s}$ with respect to the operator $\underline{A}_{0,l}^\epsilon$ that satisfies:
\[
    (\underline{A}_{0,l}^\epsilon\underline{P}_{{A}_{0,l}^\epsilon}^s\underline{\mathbb{u}}_l,\;\underline{\mathbb{v}}_{l}^s)_{0,l}
    =  (\underline{A}_{0,l}^\epsilon\underline{\mathbb{u}}_l,\;\underline{\mathbb{v}}_{l}^s)_{0,l}, \quad
    \forall \underline{\mathbb{u}}_l \in\underline{\mathbb{V}}_{l,0}^0,\;
    \underline{\mathbb{v}}_{l}^s\in\underline{\mathbb{V}}_{l,0}^{0,s},\;
    \forall s\in\EuScript{S}_l.
\]
The damped block Jacobi smoother is then expressed as:
\begin{align}
\label{bjac}
    \underline{R}_l:=
    \varsigma\sum_{s\in \EuScript{S}_l}\underline{P}_{{A}_{0,l}^\epsilon}^s (\underline{A}_{0,l}^{\epsilon})^{-1},
\end{align}
where the damping parameter $\varsigma > 0$ is small enough to ensure 
that the operator $(id - \underline{R}_l\underline{A}_{0,l}^\epsilon)$
is a positive definite contraction, and only depends on the bounded number of overlapping blocks \cite{arnold2000multigrid}. We further define $\underline{R}_l^T$ as the transpose of $\underline{R}_l$ with respect to the inner product $(\cdot,\;\cdot)_{0,l}$.

Now we are ready to present the W-cycle and variable V-cycle multigrid algorithms for the linear system $\underline{A}_{0,l}^\epsilon \underline{\mathbb{u}}_l = \underline{\mathbb{g}}_l\in \underline{\mathbb{V}}_{l,0}^0$ 
as in Algorithm \ref{alg:st-hMG}.
\begin{algorithm}[ht]
\caption{The $h$-multigrid algorithm at lowest order.}
\label{alg:st-hMG}
\begin{algorithmic}
\State The $l$-th level multigrid algorithm produces $MG_{h}(l, \underline{\mathbb{g}}_l, \underline{\mathbb{u}}_l^{(0)}, m(l), q)$
as an approximation solution for
$\underline{A}_{0,l}^\epsilon\underline{\mathbb{u}}_l = \underline{\mathbb{g}}_l$ with initial guess $\underline{\mathbb{u}}_l^{(0)}$. Here $m(l)$ denotes the number of pre-smoothing and post-smoothing steps on the $l$-th mesh level, $q=1$ corresponds to the V-cycle algorithm, and $q=2$ corresponds to the W-cycle algorithm where $m(1)=\cdots=m(l)=m$.
\If {$l = 1$}
    \State
    ${MG}_{h}(l, \underline{g}_l,  \underline{\mathbb{u}}_l^{(0)}, m(l), q) = (\underline{A}_{0,l}^\epsilon)^{-1} \underline{\mathbb{g}}_l$.
\Else
\State Perform the following three steps:
\State (1) {\it Pre-smoothing.}
For $j = 1,\dots,m(l)$, compute $\underline{\mathbb{u}}_l^{(j)}$ by
    \[\underline{\mathbb{u}}_l^{(j)} = 
    \underline{\mathbb{u}}_l^{(j-1)} + \underline{R}_l(\underline{\mathbb{g}}_l - \underline{A}_{0,l}^\epsilon \underline{\mathbb{u}}_l^{(j-1)}).\]
\State (2) {\it Coarse-grid correction.}
Let $\delta{\underline{\mathbb{u}}}_{l-1}^{(0)} = \underline{\mathbb{0}}$ and ${\underline{\mathbb{r}}}_{l-1}= \underline{\mathcal{I}}_l^{l-1}(\underline{\mathbb{g}}_l - \underline{A}_{0,l}^\epsilon \underline{\mathbb{u}}_l^{m(l)})$. For $n = 1,\dots,q$, \indent compute $\delta{\underline{\mathbb{u}}}_{l-1}^{(n)}$ by
\[
    \delta{\underline{\mathbb{u}}}_{l-1}^{(n)} =  
    MG_{h}(l-1, \underline{\mathbb{r}}_{l-1}, \delta{\underline{\mathbb{u}}}_{l-1}^{(n-1)}, m(l-1), q). 
\]
\indent Then we get $\underline{\mathbb{u}}_l^{(m(l)+1)} = \underline{\mathbb{u}}_l^{(m(l))} + \underline{\mathcal{I}}_{l-1}^l \delta{\underline{\mathbb{u}}}_{l-1}^{(q)}$.
\State (3) {\it Post-smoothing.} 
For $j=m(l)+2,\dots, 2m(l)+1$, compute $\underline{\mathbb{u}}_l^{(j)}$
by 
\[
    \underline{\mathbb{u}}_l^{(j)} = 
    \underline{\mathbb{u}}_l^{(j-1)} + \underline{R}_l^T(\underline{\mathbb{g}}_l - \underline{A}_{0,l}^\epsilon \underline{\mathbb{u}}_l^{(j-1)}),
\]
\State We then define 
$MG_h(l, \underline{\mathbb{g}}_l,\underline{\mathbb{u}}_l^0, m(l), q) = \underline{\mathbb{u}}_l^{(2m(l)+1)}$.\EndIf
\end{algorithmic}
\end{algorithm}

Sch\"oberl's multigrid theory \cite{schoberl1999multigrid, schoberl1999robust}, originally introduced for the $\mathrm{P}^2$-$\mathrm{P}^0$ discretization on triangles, can be applied to the CR discretization in two- and three-dimensions and requires full elliptic regularity results in \eqref{stReg}. The proof procedures for the optimality of W-cycle and variable V-cycle multigrid for the CR discretization with pressure-robust treatment \eqref{equivCRstWeak} are essentially the same as in our previous work \cite{fk2022optimal}, where we used Sch\"oberl's theory for the CR discretization without pressure-robust treatment, and we refer to \cite[Remark 4.4]{fk2022optimal} for more details. We note that the only difference of the left-hand-side operator before and after the pressure-robust treatment lies in the mass term. The $L^2$ norm of $\underline{\Pi}^{RT}\underline{v}_h^{CR}$ is bounded by $\|\underline{v}_h^{CR}\|_0$ for any $\underline{v}_h^{CR}\in\underline{V}_{h}^{CR}$.
Here we quote the optimality result from \cite[Theorem 3.7]{schoberl1999robust}, from which we have, for the lowest-order case, W-cycle multigrid is a convergent iteration method when the smoothing steps are large enough, and V-cycle multigrid is a preconditioner, with both methods robust concerning the mesh size and the penalty parameter $\epsilon$.

\begin{theorem}[Theorem 3.7 of \cite{schoberl1999robust}]
\label{thm:st-hMG}
The geometric $h$-multigrid procedure defined in Algorithm \ref{alg:st-hMG} has the following properties:
\begin{itemize}
    \item The W-cycle multigrid algorithm is a robust convergent method. Specifically, there exist positive constants $m_\ast$ and $C$, independent of the mesh size $h_l$ and the penalty parameter $\epsilon$, such that with $q = 2$ and $m(1)= \cdots=m(l) = m$ in Algorithm \ref{alg:st-hMG}, we have
    \[
        \|\underline{\mathbb{E}}_{l,m}\underline{\mathbb{v}}_l\|_{\underline{A}_{0,l}^\epsilon}
        \le  C m^{-1/4} \|\underline{\mathbb{v}}_l\|_{\underline{A}_{0,l}^\epsilon},
        \quad
        \forall \underline{\mathbb{v}}_l\in\underline{\mathbb{V}}_{l,0}^0,\;
        l\ge 1,\; m\ge m_\ast,
    \]
    where ${\underline{\mathbb{E}}}_{l,m}: \underline{\mathbb{V}}_{l,0}^0\rightarrow \underline{\mathbb{V}}_{l,0}^0$
    is the operator relating the initial error and the final error of the $W$-cycle multigrid algorithm, i.e., 
    \[
    {\underline{\mathbb{E}}}_{l,m}(\underline{\mathbb{u}}_l-\underline{\mathbb{u}}_l^{(0)}) := \underline{\mathbb{u}}_l - MG_h(l, \underline{\mathbb{g}}_l, \underline{\mathbb{u}}_l^{(0)}, m, 2). 
    \]
    \item The variable V-cycle algorithm is a robust preconditioner. Specifically, with $q=1$ and $\gamma_0 m(l) \le m(l-1) \le \gamma_1 m(l)$ in Algorithm \ref{alg:st-hMG} ($1<\gamma_0<\gamma_1$), there exists a positive constant $C$ independent of the mesh size $h_l$ and the penalty parameter $\epsilon$ such that
    \[
        \kappa(\underline{\mathbb{B}}_{l,m(l)}\underline{A}_{0,l}^\epsilon) \le 1 + C m(l)^{-1/4} ,
    \]
    where $\kappa$ is the condition number, and  ${\underline{\mathbb{B}}}_{l,m(l)}: \underline{\mathbb{V}}_{l,0}^0\rightarrow \underline{\mathbb{V}}_{l,0}^0$
    is the preconditioning operator relating the residual to the correction of the variable $V$-cycle multigrid algorithm with a zero initial guess, i.e., 
    \[
        \underline{\mathbb{B}}_{l,m(l)}\underline{\mathbb{v}}_l:=
        MG_h(l, \underline{\mathbb{v}}_l, 0, m(l),1),
        \quad \forall \underline{\mathbb{v}}_l\in\underline{\mathbb{V}}_{l,0}^0.
    \]
\end{itemize}
\end{theorem}

\subsection{$hp$-Multigrid algorithm for the higher-order scheme}
We now consider the augmented Lagrangian Uzawa iteration \eqref{stOpEq-ALU} for the condensed higher-order $H$(div)-HDG scheme with $k\geq 1$. We present an $hp$-multigrid algorithm to precondition the primal operator $\underline{A}_{k,h}^\epsilon$ on the finest mesh $\Th=\mathcal{T}_J$.  Specifically, we define a "two-level" nested ASP for the system $\underline{A}_{k,h}^\epsilon\underline{\mathbb{u}}_h^\partial=\underline{\mathbb{g}}_h^\partial \in\underline{\mathbb{V}}_{h,0}^{k,\partial}$, where we use the lowest-order velocity space $\underline{\mathbb{V}}_{h,0}^{0}$ as the auxiliary space and $MG_h$ in Algorithm \ref{alg:st-hMG} as the inexact auxiliary space solver.

Two components are needed to finish constructing the multiplicative ASP: (i) a prolongation operator to transfer functions from the lowest-order velocity space $\underline{\mathbb{V}}_{h,0}^0$ to the higher-order global velocity space $\underline{\mathbb{V}}_{h,0}^{k,\partial}$, and (ii) a relaxation method to reduce the errors in the high-order spaces.
For the prolongation operator, we use the natural inclusion operator denoted by $\underline{\Pi}_{0}^k$. The transpose of this operator with respect to $(\cdot,\;\cdot)_{0,h}$ is the $L^2$-projection operator $\underline{\Pi}_{k}^0$.
For the relaxation method, we again need to take care of the divergence-free kernels in $\underline{\mathbb{V}}_{h,0}^{k,\partial}$ and we use vertex-patched block Jacobi/Gauss-Seidel method.
With slight abuse of notation, for any vertex $s\in\EuScript{S}_h=\EuScript{S}_J$, we denote $\underline{\mathbb{V}}_{h,0}^{k,\partial,s}$ as the subspace of $\underline{\mathbb{V}}_{h,0}^{k,\partial}$ with support on the vertex-patched mesh elements $\mathcal{T}_h^s$ and vertex-patched mesh skeletons $\mathcal{E}_{h}^s$, and $\underline{P}^s_{A_{k,h}^\epsilon}:\underline{\mathbb{V}}_{h,0}^{k,\partial}\rightarrow\underline{\mathbb{V}}_{h,0}^{k,\partial,s}$ as the $\underline{A}_{k,h}^\epsilon$-orthogonal projection onto vertex-patched subspace ${\mathbb{V}}_{h,0}^{k,\partial,s}$. 
We define $\underline{R}_h^\partial$ as the corresponding damped block Jacobi smoother
\[
    \underline{R}_h^\partial := \varsigma'\sum_{s\in\mathcal{S}_h}\underline{P}_{A_{k,h}^\epsilon}^s (\underline{A}_{k,h}^\epsilon)^{-1},
\]
where $\varsigma'>0$ is the damping parameter, and $\underline{R}_h^{\partial, T}$ as the transpose of $\underline{R}_h^\partial$ with respect to the inner product $(\cdot,\;\cdot)_{0,h}$. Now we present the $hp$-multigrid method for $\underline{A}_{k,h}^\epsilon\underline{\mathbb{u}}_h^\partial=\underline{\mathbb{g}}_h^\partial$ in Algorithm \ref{alg:st-hpMG}.

\begin{algorithm}[ht]
\caption{The $hp$-multigrid algorithm for $\underline{A}_{k,h}^\epsilon\underline{\mathbb{u}}_h^\partial=\underline{\mathbb{g}}_h^\partial$.}
\label{alg:st-hpMG}
\begin{algorithmic}
\State The $hp$-multigrid algorithm produces $MG_{hp}(\underline{\mathbb{g}}_h^\partial, \underline{\mathbb{u}}_h^{\partial (0)}, m_p, m_h, q)$
as an approximation solution for
$\underline{A}_{k,h}^\epsilon\underline{\mathbb{u}}_h^\partial=\underline{\mathbb{g}}_h^\partial$ with initial guess $\underline{\mathbb{u}}_h^{\partial (0)}$. Here $m_p$ denotes the number of pre-relaxation and post-relaxation steps, $m_h$ denotes the number of pre-smoothing and post-smoothing steps of the lowest-order geometric multigrid $MG_h$ on the finest mesh level, and $q\in\{1,2\}$ decides if $MG_h$ is W-cycle ($q=2$) or variable V-cycle ($q=1$).
\If {$k = 0$}
    \State
    ${MG}_{hp}(\underline{g}_h^\partial,  \underline{\mathbb{u}}_h^{\partial (0)}, m_p, m_h, q) = MG_h(J, \underline{g}_h^\partial, \underline{\mathbb{u}}_h^{\partial (0)}, m_h, q)$.
\Else
\State Perform the following three steps:
\State (1) {\it Pre-relaxation.}
For $j = 1,\dots,m_p$, compute $\underline{\mathbb{u}}_h^{\partial (j)}$ by
    \[\underline{\mathbb{u}}_h^{\partial (j)} = 
    \underline{\mathbb{u}}_h^{\partial (j-1)} + \underline{R}_h^\partial(\underline{\mathbb{g}}_h^\partial - \underline{A}_{k,h}^\epsilon \underline{\mathbb{u}}_h^{\partial (j-1)}).\]
\State (2) {\it Lowest-order space correction.}
Let ${\underline{\mathbb{r}}}_{h}^0= \underline{\Pi}^0_k(\underline{\mathbb{g}}_h^\partial - \underline{A}_{k,h}^\epsilon \underline{\mathbb{u}}_h^{\partial (m_p)})$. Compute $\delta{\underline{\mathbb{u}}}_{h}^0$ by
\[
    \delta{\underline{\mathbb{u}}}_{h}^0 =  
    MG_{h}(J, \underline{\mathbb{r}}_{h}^0, \underline{\mathbb{0}}, m_h, q). 
\]
\indent Then we get $\underline{\mathbb{u}}_h^{\partial (m_p+1)} = \underline{\mathbb{u}}_h^{\partial (m_p)} + \underline{\Pi}_{0}^k \delta{\underline{\mathbb{u}}}_{h}^0$.
\State (3) {\it Post-relaxation.} 
For $j=m_p+2,\dots, 2m_p+1$, compute $\underline{\mathbb{u}}_h^{\partial (j)}$
by 
\[
    \underline{\mathbb{u}}_h^{\partial (j)} = 
    \underline{\mathbb{u}}_h^{\partial (j-1)} + \underline{R}_h^{\partial, T}(\underline{\mathbb{g}}_h - \underline{A}_{k,h}^\epsilon \underline{\mathbb{u}}_h^{\partial (j-1)}),
\]
\State We then define 
$MG_{hp}(\underline{\mathbb{g}}_h^\partial,\underline{\mathbb{u}}_h^{\partial (0)}, m_p, m_h, q) = \underline{\mathbb{u}}_h^{\partial (2m_p+1)}$.\EndIf
\end{algorithmic}
\end{algorithm}

\begin{remark}[On the $hp$-multigrid method] Our goal is to achieve a mild increase or robustness in iteration counts of the Krylov space solver preconditioned by the $hp$-multigrid method as the polynomial order increases. Our $hp$-multigrid algorithm can be seen as a Schwarz-type method \cite{XuSubspace92}, utilizing a space decomposition given by
\begin{align}
    \label{hpSubspace}
    \underline{\mathbb{V}}_{h,0}^{k,\partial} = \underline{\mathbb{V}}_{h,0}^{0} + \sum_{s\in\mathcal{S}_h}\underline{\mathbb{V}}_{h,0}^{k,\partial,s}.
\end{align}
Based on the same subspace decomposition (lowest-order space $+$ vertex-patched subspaces), Pavarino  \cite{pavarino94} introduced and proved a $p$-robust additive Schwarz preconditioner for the continuous finite element space for Poisson's equation, and Brubeck and Farrell \cite{Brubeck22} further scale this $p$-robust preconditioner to very high polynomial degrees. For a nearly-singular system such as in \eqref{stOpEq-ALU} where $\epsilon$ approaches 0, Lee et al. \cite{lee2007robust} presented a general framework and proved that the method of subsequent subspace correction (MSSC) is convergent provided that the exact/inexact solver on each subspace is a robust contraction and that the kernel of the divergence operator can be decomposed into sum of elements of the subspaces, and the convergence rate is robust with respect to mesh size and parameter $\epsilon$. Our numerical experiments show that the $hp$-multigrid in Algorithm \ref{alg:st-hpMG} is robust with respect to mesh size and parameter $\epsilon$, and the preconditioned Krylov space solver has only a very mild increase in iteration counts as the polynomial order increases.

We acknowledge that our current $hp$-multigrid method is multiplicative and subsequent between the lowest order spaced $\underline{\mathbb{V}}_{h,0}^{0}$ and the vertex-patched subspaces $\sum_{s\in\mathcal{S}_h}\underline{\mathbb{V}}_{h,0}^{k,\partial,s}$, whereas ASP introduced in \cite{Xu96} is additive and parallel. However, we observed drastic increases in iteration counts when $\epsilon\rightarrow 0$ and the polynomial order increases when an additive version of the $hp$-multigrid is used as the preconditioner in our numerical experiments.
Therefore, a parallel and robust $hp$-multigrid for the operator $\underline{A}_{k,h}^\epsilon$ along with its theoretical proof is worth pursuing and will be the focus of our future research.
\end{remark}

\section{Navier-Stokes equation}
\label{sec:ns}
\subsection{Model problem and the $H$(div)-HDG scheme}
The model problem is to find $(\underline{u}, p)$ satisfying 
\begin{subequations}
    \label{nsModel}
    \begin{alignat}{2}
        \beta \underline{u} - \nabla\cdot(\nu\nabla\underline{u}) + \underline{u}\cdot\nabla\underline{u}
        + \nabla p 
        =& \;\underline{f}, \quad &&\text{in $\Omega$,}
        \\
        \nabla\cdot\underline{u} =& \; 0, \quad&&\text{in $\Omega$,}
        \\
        \underline{u}=& \;\underline{0}, \quad &&\text{on $\partial\Omega$,}
    \end{alignat}
\end{subequations}
with notations the same as in \eqref{stModel}.
We introduce the tensor $\dunderline{L}:= - \nu\nabla\underline{u}$ as a new variable and rewrite \eqref{nsModel} into a first-order system:
\begin{subequations}
    \label{nsModelMixed}
    \begin{alignat}{2}
        \nu^{-1}\dunderline{L} + \nabla\underline{u} =& \; 0, \quad&&\text{in $\Omega$,}
        \\
        \beta\underline{u} + \nabla\cdot\dunderline{L} + \underline{u}\cdot\nabla\underline{u}
        + \nabla p =& \; \underline{f}, \quad&&\text{in $\Omega$,}
        \\
        \nabla\cdot\underline{u} =& \; 0, \quad&&\text{in $\Omega$,}
        \\
        \underline{u} =& \; \underline{0}, \quad&&\text{on $\partial\Omega$.}
    \end{alignat}
\end{subequations}

For the nonlinear convection term $\underline{u}\cdot\nabla\underline{u}$ in the Navier-Stokes equations, we use the natural upwind discretization which needs no additional stabilization and leads to minimal numerical dissipation \cite[Chapter 2]{lehrenfeld2010hybrid}. The trilinear form for the $H$(div)-HDG discretization for the convection term is defined as
\begin{align*}
    \underline{\mathcal{C}}_h(\underline{w}_h;\;\underline{\mathbb{u}}_h,\; \underline{\mathbb{v}}_h):= &
            - (\underline{w}_h \otimes \underline{u}_h,\; \nabla\underline{v}_h)_{\Th}
            +\langle (\underline{w}_h\cdot\underline{n}_K)\underline{u}_h^{up},\; \mathsf{tng}(\underline{v}_h - \widehat{\underline{v}}_h)\rangle_{\partial \Th}, 
\end{align*}
for all $\underline{w}_h\in\underline{V}_{h,0}^{k}$ and $\underline{\mathbb{u}}_h,\underline{\mathbb{v}}_h\in\underline{\mathbb{V}}_{h,0}^{k}$, where
\begin{align*}
    \underline{u}_h^{up}:=& \mathsf{nrm}(\underline{u}_h) + \left\{
        \renewcommand{\arraystretch}{1.5} 
        \begin{array}{ll}
            \mathsf{tng}(\underline{u}_h),& \text{if $\underline{u}_h\cdot\underline{n}_K > 0$},  \\[1ex]
            \mathsf{tng}(\widehat{\underline{u}}_h),& \text{if $\underline{u}_h\cdot\underline{n}_K < 0$}.
        \end{array}
    \right.
\end{align*}

\subsection{Linearization and iterative solving procedure}
We linearize the convection term by Picard or Newton's method and then iteratively solve the resulting linearized $H$(div)-HDG scheme. Given $\underline{\mathbb{u}}_{h}^{(n-1)}\in\underline{\mathbb{V}}_{h,0}^k$ at the previous step, when Picard iteration is used, the linearized convection term at the $n$-th step is
\[
     \underline{\mathcal{C}}_h^l(\underline{u}_h^{(n-1)};\;\underline{\mathbb{u}}_h^{(n)},\; \underline{\mathbb{v}}_h):=
     \underline{\mathcal{C}}_h(\underline{u}_h^{(n-1)};\;\underline{\mathbb{u}}_h^{(n)},\; \underline{\mathbb{v}}_h),
\]
where $\underline{\mathbb{u}}_h^{(n)}\in\underline{\mathbb{V}}_{h,0}^k$ is the velocity solution to be found at the $n$-th step, and when Newton iteration is used, the linearized convection term at $n$-th step becomes
\[
    \underline{\mathcal{C}}_h^l(\underline{u}_h^{(n-1)};\;\delta\underline{\mathbb{u}}_h^{(n)},\; \underline{\mathbb{v}}_h)
    :=
    \underline{\mathcal{C}}_h(\underline{u}_h^{(n-1)};\;\delta\underline{\mathbb{u}}_h^{(n)},\; \underline{\mathbb{v}}_h)
    + \underline{\mathcal{C}}_h(\delta\underline{u}_h^{(n)};\;\underline{\mathbb{u}}_h^{(n-1)},\; \underline{\mathbb{v}}_h),
\]
where $\delta\underline{\mathbb{u}}_h^{(n)}\in\underline{\mathbb{V}}_{h,0}^k$ is the change to the velocity at the $n$-th step.

Then given solution $(\dunderline{L}_h^{(n-1)}, \underline{\mathbb{u}}_h^{(n-1)}, p_h^{(n-1)})$ at the previous step, the linearized $H$(div)-HDG scheme is to find $(\dunderline{L}_h, \underline{\mathbb{u}}_h, p_h) \in\dunderline{W}_h^k \times\underline{\mathbb{V}}_{h,0}^k \times W_{h,0}^{k}$, $k \geq 0$, such that
\begin{subequations}
    \label{nsWeak}
    \begin{align}
    \label{nsWeak1}
        (\nu^{-1}\dunderline{L}_h,\; \dunderline{G}_h)_{\mathcal{T}_h}  
        + (\nabla\underline{u}_h,\; \dunderline{G}_h)_{\mathcal{T}_h}
        - \langle \mathsf{tng}(\underline{u}_h - \widehat{\underline{u}}_h),\; \dunderline{G}_h\underline{n}_K\rangle_{\partial\mathcal{T}_h} &= 0,
        \\
    \label{nsWeak2}
        (\beta\underline{u}_h,\; \underline{v}_h)_{\mathcal{T}_h}
        -( \dunderline{L}_h,\; \nabla\underline{v}_h )_{\mathcal{T}_h}
        + \langle \dunderline{L}_h\underline{n}_K, \mathsf{tng}(\underline{v}_h - \widehat{\underline{v}}_h) \rangle_{\partial\Th} &
        \\ \nonumber
        + \underline{\mathcal{C}}_h^l(\underline{u}_h^{(n-1)};\;\underline{\mathbb{u}}_h,\; \underline{\mathbb{v}}_h)
        - (p_h, \nabla\cdot\underline{v}_h)_{\Th}
        & = (\underline{f}',\; \underline{v}_h)_{\mathcal{T}_h}
        \\
    \label{nsWeak3}
        (\nabla\cdot\underline{u}_h,\; q_h)_{\mathcal{T}_h}
        &= 0.
    \end{align}
\end{subequations}
for all $(\dunderline{G}_h, \underline{\mathbb{v}}_h, q_h) \in\dunderline{W}_h^k \times\underline{\mathbb{V}}_{h,0}^k \times W_{h,0}^{k}$, where for the Picard method,
\[
    (\underline{f}',\; \underline{v}_h)_{\mathcal{T}_h} := (\underline{f},\; \underline{v}_h)_{\mathcal{T}_h},
\]
and for the Newton's method,
\begin{align*}
    (\underline{f}',\; \underline{v}_h)_{\mathcal{T}_h} := &
    (\underline{f},\; \underline{v}_h)_{\mathcal{T}_h}
    -(\nu^{-1}\dunderline{L}_h^{(n-1)},\; \dunderline{G}_h)_{\mathcal{T}_h}  
    -(\nabla\underline{u}_h^{(n-1)},\; \dunderline{G}_h)_{\mathcal{T}_h}
    + \langle \mathsf{tng}(\underline{u}_h^{(n-1)} - \widehat{\underline{u}}_h^{(n-1)}),\; \dunderline{G}_h\underline{n}_K\rangle_{\partial\mathcal{T}_h}
    \\&
    -(\beta\underline{u}_h^{(n-1)},\; \underline{v}_h)_{\mathcal{T}_h}
    +( \dunderline{L}_h^{(n-1)},\; \nabla\underline{v}_h )_{\mathcal{T}_h}
    - \langle \dunderline{L}_h^{(n-1)}\underline{n}_K, \mathsf{tng}(\underline{v}_h - \widehat{\underline{v}}_h) \rangle_{\partial\Th}
    \\&
    - \underline{\mathcal{C}}_h^l(\underline{u}_h^{(n-1)};\;\underline{\mathbb{u}}_h^{(n-1)},\; \underline{\mathbb{v}}_h)
    + (p_h^{(n-1)}, \nabla\cdot\underline{v}_h)_{\Th}.
\end{align*}
We denote the operator form of the condensed global system of the linearized $H$(div)-HDG scheme \eqref{nsWeak} as to find $(\underline{\mathbb{u}}_h^\partial, p_h^{\partial})\in\underline{\mathbb{V}}_{h,0}^{k,\partial}\times W_{h,0}^{\partial}$ satisfying:
\begin{subequations}
    \label{nsOptForm}
    \begin{align}
        \underline{A}_{k,h}^c\underline{\mathbb{u}}_h^\partial
        + \underline{B}_{k,h}^\ast p_h^\partial =\;& \underline{F}_{k,h}^c,
        \\
        \underline{B}_{k,h} \underline{\mathbb{u}}_h^\partial =\;& 0,
\end{align}
\end{subequations}
where the operators $\underline{B}_{k,h}$ and $\underline{B}_{k,h}^\ast$ are identical to those in \eqref{stOptForm}. The operator $\underline{A}_{k,h}^c$ is obtained by augmenting $\underline{A}_{k,h}$ with a condensed operator from the linearized convection term $\underline{\mathcal{C}}_h^l$. To solve the resulting saddle-point system, we employ an augmented Lagrangian Uzawa iteration, similar to the approach used for the generalized Stokes equation: With $p_h^{\partial (0)} = 0$, iteratively find solution $(\underline{\mathbb{u}}_h^{\partial(n)},p_h^{\partial(n)})\in\underline{\mathbb{V}}_{h,0}^{k,\partial}\times W_{h,0}^{\partial}$ that satisfies
\begin{subequations}
    \label{nsOpEq-ALU}
    \begin{align}
    \label{nsOpEq-ALU1}
    \underline{A}_{k,h}^{c, \epsilon} \underline{\mathbb{u}}_h^{\partial\,(n)} =& \;\underline{F}_{k,h}^c - \underline{B}_{k,h}^\ast p_h^{\partial\,(n-1)},
        \\
    \label{nsOpEq-ALU2}
    p_h^{\partial\, (n)} =& \;   p_h^{\partial\, (n-1)}- \epsilon^{-1}  \underline{B}_{k,h} \underline{\mathbb{u}}_h^{\partial\, (n)},
\end{align}
\end{subequations}
where 
\[
    \underline{A}_{k,h}^{c, \epsilon}:=
    \underline{A}_{k,h}^c + \epsilon^{-1}\underline{B}_{k,h}^\ast\underline{B}_{k,h},
\]
with the penalty parameter $\epsilon^{-1}\gg 1$.

When the viscosity $\nu$ approaches zero, the convection term in the Navier-Stokes equations dominates over the diffusion term, and this dominance causes the linear system to become more ill-conditioned, making it challenging to find a robust preconditioner for Krylov subspace solvers such as GMRes \cite{gmres}. Moreover, the non-symmetry of the $H$(div)-HDG scheme further complicates the analysis of the preconditioner's robustness.
However, in a study by Benzi and Olshanskii \cite{benziOlshanskii06}, Sch\"oberl's geometric multigrid method \cite{schoberl1999multigrid, schoberl1999robust} was applied to precondition the primal operator of the augmented Lagrangian formulation of the two-dimensional Oseen problem discretized by $\mathrm{iso}\mathrm{P}^2$-$\mathrm{P}^0$ and $\mathrm{iso}\mathrm{P}^2$-$\mathrm{P}^1$ elements with streamline-upwind Petrov-Galerkin (SUPG) stabilization. Numerical experiments demonstrate that this approach is essentially robust with respect to the Reynolds number.
Subsequently, Farrell et al. \cite{farrell2019augmented} extended this approach to the three-dimensional Newton linearized Navier-Stokes equations, which is discretized by a $\mathrm{P}^1$-$\mathrm{P}^0$ pair with velocity space enriched by facet bubbles and with SUPG stabilization. Numerical experiments support the preconditioner's robustness with respect to the Reynolds number.
Motivated by these results, we adapt the $hp$-multigrid algorithm developed for the generalized Stokes equation in Algorithm \ref{alg:st-hpMG} to precondition the operator $\underline{A}_{k,h}^{c, \epsilon}$ of the linearized Navier-Stokes equations in \eqref{nsOpEq-ALU1}. We employ the same block relaxation method, block smoothing method, and intergrid transfer operators with discrete harmonic extensions.

It is important to note that the Newton iteration solver has quadratic convergence, but it requires a sufficiently accurate initial guess. To address this and test the performance of the proposed $hp$-multigrid for both the Picard and the Newton linearization, the Picard iteration method is firstly utilized in this study to solve the Navier-Stokes equations until the $L^2$ norm of the difference between the velocities of two consecutive steps becomes smaller than a predefined tolerance $\epsilon_{picard}$. Subsequently, the Newton iteration method is employed, using the solution of the last Picard iteration as the initial guess for the Newton iteration. The algorithm is described in Algorithm \ref{alg:ns}. In both the Picard and Newton iteration, the result of the previous iteration is used as the initial guess of the current iteration. Fig. \ref{fig:ns-hp-MG} provides a graphical illustration of the overall solving procedure.

\begin{algorithm}[ht]
\caption{The iterative solving procedure for Navier-Stokes equations.}
\label{alg:ns}
\begin{algorithmic}
\State Obtain initial guess $\underline{\mathbb{u}}_h^{(0)}$ by solving $\underline{\mathbb{u}}_h^{(0)}$ from the generalized Stokes equation which neglects convection term and keeps all other model parameters the same.
\State Perform the following steps:
\State {\bf (i) Picard iteration:}
\While{$\|\underline{u}_h^{(n)} - \underline{u}_h^{(n-1)}\|_0 \ge \epsilon_{picard}$}
\State For $n \ge 1$, solve $(\dunderline{L}_h^{(n)},\underline{\mathbb{u}}_h^{(n)},p_h^{(n)})$ from the Picard linearized $H$(div)-HDG scheme \eqref{nsWeak}
\EndWhile
\State Redefine $\underline{\mathbb{u}}_h^{(0)}=\underline{\mathbb{u}}_h^{(n_{picard})}$, where $n_{picard}$ is the number of performed Picard iterations.
\State {\bf (ii) Newton iteration:}
\While{$\|\underline{u}_h^{(n)} - \underline{u}_h^{(n-1)}\|_0 \ge \epsilon_{newton}$}
\State For $n \ge 1$, solve $(\delta\dunderline{L}_h^{(n)}, \delta\underline{\mathbb{u}}_h^{(n)}, \delta p_h^{(n)})$ from the Newton linearized $H$(div)-HDG scheme \eqref{nsWeak}.
\State Then we have $(\dunderline{L}_h^{(n)}, \underline{\mathbb{u}}_h^{(n)}, p_h^{(n)}) = (\dunderline{L}_h^{(n-1)}, \underline{\mathbb{u}}_h^{(n-1)}, p_h^{(n-1)}) + (\delta\dunderline{L}_h^{(n)}, \delta\underline{\mathbb{u}}_h^{(n)}, \delta p_h^{(n)})$
\EndWhile
\end{algorithmic}
\end{algorithm}

\begin{figure}[ht]
    \centering
    \includegraphics[width=1\textwidth]{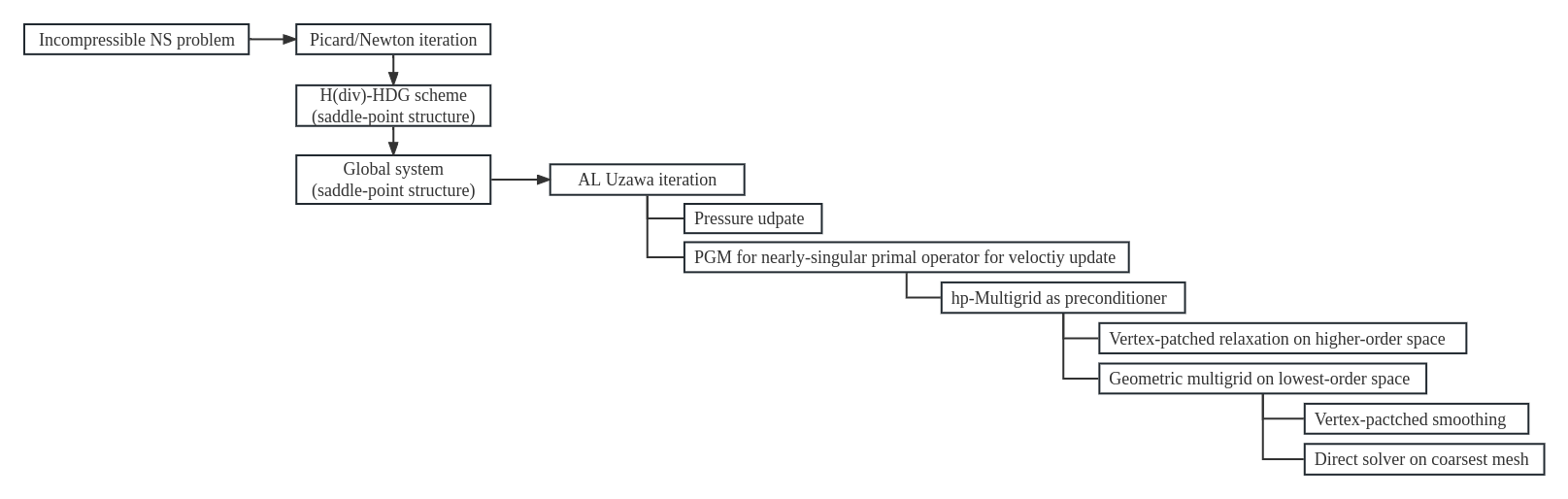}
    \caption{Graphic illustration of the Algorithm \ref{alg:ns} process.}
    \label{fig:ns-hp-MG}
\end{figure}

\section{Numerical Experiments}
\label{sec:num}
This section presents the numerical experiments carried out to validate the optimal convergence rates of the $H(\text{div})$-HDG scheme and the robustness of the proposed $hp$-multigrid method. We conduct experiments for both the generalized Stokes equation and the Navier-Stokes equations. For the latter, we solve the steady cases with $\beta = 0$.
To solve the condensed $H(\text{div})$-HDG scheme, we use two-step augmented Lagrangian Uzawa iteration method with $(\nu\epsilon)^{-1}=10^6$ for both set of equations. In updating global velocity in each step of Uzawa iteration, the preconditioned conjugate gradient solver (PCG) is adopted for the generalized Stokes equation, while the preconditioned GMRes solver (PGM) is adopted for the linearized Navier-Stokes equations. The proposed $hp$-multigrid method is employed as the preconditioner for both solvers, and the stopping criterion is a relative tolerance of $10^{-8}$ and an absolute tolerance of $10^{-10}$.
In Algorithm \ref{alg:ns}, we set $\varepsilon_{\text{picard}}$ to $10^{-4}$, and the Newton iteration stops with a relative tolerance of $10^{-8}$ and an absolute tolerance of $10^{-10}$.
In all cases, we use block Gauss-Seidel method for both the relaxation and smoothing to avoid the damping parameter in Jacobi method. We further set the relaxation steps to be the same as the smoothing steps, i.e. $m_h = m_p = m$ in the $hp$-multigrid in Algorithm \ref{alg:st-hpMG}.
All results are obtained by using the NGSolve \cite{Schoberl16} and ParaView \cite{ayachit2015paraview}.
Source code for the numerical experiments is available at \url{https://github.com/WZKuang/MG4HdivHDG}.

\subsection{Convergence rate check}
We first verify the optimal convergence rates of the $H$(div)-HDG scheme solved with the proposed $hp$-multigrid method for the generalized Stokes equation and the Navier-Stokes equations with manufactured solutions.
We set the exact solution 
\begin{align*}
	&\left.
		\begin{array}{l l}
			u_x & = x^2 (x - 1)^2 2y(1 - y)(2y - 1) \\
			u_y & = y^2 (y - 1)^2 2 x(x - 1)(2x - 1)\\
			p & = x(1 - x)(1 - y) - 1/12
		\end{array}
	\right\} \;\text{when $d = 2$,}
\end{align*}
and
\begin{align*}
	&\left.
		\begin{array}{l l}
			u_x & = x^2(x - 1)^2 (2 y - 6y^2 + 4 y^3)(2 z - 6z^2 + 4 z^3) \\
			u_y & = y^2 (y - 1)^2 (2x - 6x^2 + 4x^3) (2z - 6z^2 + 4z^3)\\
			u_z & = -2 z^2 (z - 1)^2 (2x - 6x^2 + 4x^3) (2y - 6y^2 + 4y^3)\\
			p & = x(1 - x)(1 - y)(1 - z) - 1/24
		\end{array}
	\right\} \;\text{when $d = 3$.}
\end{align*}
The source term $\underline{f}$ is obtained by plugging the exact solution into the model problem.
\subsubsection{Generalized Stokes equation}
\label{subsec:stokesEocNum}
We set the domain as a unit square/cube $\Omega=[0, 1]^d$ with homogeneous Dirichlet boundary conditions on all sides.
The coarsest mesh is a triangulation of $\Omega$ with the maximum element diameter $1/h=2$ in 2D or $1/h=3$ in 3D, followed by uniform refinement by connecting the midpoints of the element boundaries. After solving $(\dunderline{L}_h,\underline{u}_h, p_h)$ from the $H$(div)-HDG scheme \eqref{stWeak}, a post-processed velocity approximation $\underline{u}_h^\ast \in\underline{W}_h^{k+1}$ is obtained element-wise by solving
\begin{subequations}
    \label{stPost}
    \begin{alignat}{2}
        (\nabla\underline{u}_h^\ast,\; \nabla\underline{v}_h)_K =\;&
        (\dunderline{L}_h,\; \nabla\underline{v}_h)_K,
        \quad&&
        \forall \underline{v}_h \in \underline{\poly}^{k+1}(K),
        \\
        (\underline{u}_h^\ast,\; \underline{w}_h)_K
        =\;& (\underline{u}_h,\; \underline{w}_h)_K,
        \quad&& \forall w_h\in \underline{\poly}^0(K),
    \end{alignat}
\end{subequations}
and the superconvergence property of such post-processed $\underline{u}_h^\ast$ when $k\ge 1$ has been proved in \cite{cockburnProjection}.

We set the viscosity $\nu = 1$, and Table \ref{tab:stokesRate} reports the estimated order of convergence (EOC) of the $L_2$ norms of $\underline{e}_u:=\underline{u}-\underline{u}_h$, $\dunderline{e}_L:=\dunderline{L}-\dunderline{L}_h$, and $\underline{e}_u^\ast:=\underline{u}-\underline{u}_h^\ast$ in both two-dimensional and three-dimensional cases, together with the $L_2$ norm of the divergence error $\nabla\cdot\underline{u}_h$, for different $\beta$ and polynomial orders $k$ of the finite element spaces.
As observed, the optimal $(k+1)$-th convergence rate is obtained for both the velocity $\underline{u}_h$ and the flux $\dunderline{L}_h$ when $k\ge 0$, with the globally divergence-free constraint satisfied. When $k\ge 1$, the $(k+2)$-th convergence rate is obtained for $\underline{u}_h^\ast$. We note that the deterioration of the convergence rates when $\|\underline{e}_u\|_0$ and $\|\underline{e}_u^\ast\|_0$ approach $\mathcal{O}(10^{-10})$ is due to the round-off error caused by $\epsilon^{-1}$, as mentioned in Remark \ref{remark:ALUZ}.

{\setlength{\tabcolsep}{7pt}
\begin{longtable}[ht]{|c|c|cc|cc|cc|c|cc|}
\caption{Estimated convergence rates of the $H$(div)-HDG scheme for the generalized Stokes equation in Example \ref{subsec:stokesEocNum}.\label{tab:stokesRate}}\\ 
\hline
\multirow{2}{*}{$k$} & \multirow{2}{*}{$1/h$} & \multicolumn{2}{c|}{DOF} & \multirow{2}{*}{$\|\underline{e}_u\|_0$} & \multirow{2}{*}{EOC} & \multirow{2}{*}{$\|\dunderline{e}_L\|_0$} & \multirow{2}{*}{EOC} & \multirow{2}{*}{$\|\nabla\cdot\underline{u}_h\|_0$} & \multirow{2}{*}{$\|\underline{e}_u^\ast\|_0$} & \multirow{2}{*}{EOC}  \\*
                     &                        & $N_{total}$ & $N_{global}$   &                                          &                      &                                           &                      &                                                     &                                               &                       \endfirsthead 
\hline
\multicolumn{11}{|c|}{$d = 2$}                                                                                                                                                                                                                                                                                                              \\ 
\hline
                     &                        &           &              & \multicolumn{7}{c|}{$\beta=0$}                                                                                                                                                                                                                                   \\* 
\hline
\multirow{4}{*}{0}   & 32                     & $3.5e3$   & $1.5e3$      & $1.2e{-3}$                                 & 1.2                  & $1.5e{-2}$                                  & 1.0                  & $1.3e{-13}$                                           & $9.2e{-4}$                                      & 1.4                   \\*
                     & 64                     & $1.4e4$   & $6.0e3$      & $5.8e{-4}$                                 & 1.1                  & $7.5e{-3}$                                  & 1.0                  & $1.3e{-13}$                                           & $4.0e{-4}$                                      & 1.2                   \\*
                     & 128                    & $5.7e4$   & $2.4e4$      & $2.9e{-4}$                                 & 1.0                  & $3.7e{-3}$                                  & 1.0                  & $1.3e{-13}$                                           & $1.9e{-4}$                                      & 1.1                   \\
\hline
\end{longtable}}

\subsubsection{Stationary Navier-Stokes equations}
\label{subsec:nsEocNum}
We consider the same exact solution and settings as in Example \ref{subsec:stokesEocNum}. Upon solving $(\dunderline{L}_h, \underline{u}_h, p_h)$ using the $H$(div)-HDG scheme \eqref{nsWeak}, we locally post-process $\underline{u}_h$ as in \eqref{stPost} to achieve superconvergence. Table \ref{tab:nsRate} presents the estimated order of convergence (EOC) of the $L_2$ norms of $\underline{e}_u$, $\dunderline{e}_L$, and $\underline{e}_u^\ast$ in two-dimensional and three-dimensional cases, along with the $L_2$ norm of the divergence error $\nabla\cdot\underline{u}_h$ for different values of viscosity $\nu$ and polynomial order $k$ of the finite element spaces. The observed optimal convergence rates of $\|\underline{e}_u\|_0$, $\|\underline{e}_u^\ast\|_0$, $\|\dunderline{e}_L\|_0$, and the exact divergence-free results are similar to those obtained in the generalized Stokes equation. It is noteworthy that the convergence rates deteriorate as $\|\underline{e}_u\|_0$ and $\|\underline{e}_u^\ast\|0$ approach $\mathcal{O}(10^{-8})$, which is due to the fact that $\epsilon_{newton}$ in \ref{alg:ns} is of order $\mathcal{O}(10^{-8})$ in this example.

{
\setlength{\tabcolsep}{6.5pt}
\begin{longtable}{|c|c|cc|cc|cc|c|cc|}
\caption{Estimated convergence rates of the $H$(div)-HDG scheme for the Navier-Stokes equations in Example \ref{subsec:nsEocNum}.\label{tab:nsRate}}\\ 
\hline
\multirow{2}{*}{$k$} & \multirow{2}{*}{$1/h$} & \multicolumn{2}{c|}{DOF} & \multirow{2}{*}{$\|\underline{e}_u\|_0$} & \multirow{2}{*}{EOC} & \multirow{2}{*}{$\|\dunderline{e}_L\|_0$} & \multirow{2}{*}{EOC} & \multirow{2}{*}{$\|\nabla\cdot\underline{u}_h\|_0$} & \multirow{2}{*}{$\|\underline{e}_u^\ast\|_0$} & \multirow{2}{*}{EOC}  \\*
                     &                        & $N_{total}$ & $N_{global}$   &                                          &                      &                                           &                      &                                                     &                                               &                       \endfirsthead 
\hline
\multicolumn{11}{|c|}{$d = 2$}                                                                                                                                                                                                                                                                                                                                                                                                                                        \\ 
\hline
                     &                        &           &              & \multicolumn{7}{c|}{$\nu=1$}                                                                                                                                                                                                                                                                                                                                                                \\* 
\hline
\multirow{4}{*}{0}   & 16                     & $8.6e2$   & $3.5e2$      & $2.8e{-3}$                                                      & 1.4                                      & $2.9e{-2}$                                                       & 0.9                                      & $6.3e{-11}$                                           & $2.1e{-3}$                                                           & 1.6                                       \\*
                     & 32                     & $3.5e3$   & $1.5e3$      & $1.2e{-3}$                                                      & 1.2                                      & $1.5e{-2}$                                                       & 1.0                                      & $9.8e{-12}$                                           & $8.0e{-4}$                                                           & 1.4                                       \\
\hline
\end{longtable}}

\subsection{Lid-driven cavity problem}
In this example, we investigate the robustness of the proposed $hp$-multigrid preconditioners for the lid-driven cavity problem. The computational domain is a unit square/cube $\Omega=[0,1]^d$. We assume an inhomogeneous Dirichlet boundary condition $\underline{u}=[4x(1-x),\; 0]^\trans$ when $d=2$ or $\underline{u}=[16x(1-x)y(1-y),\; 0,\; 0]^\trans$ when $d=3$ on the top side, and no-slip boundary conditions on the remaining domain boundaries. The source term $\underline{f}=\underline{0}$. The coarsest mesh is a triangulation of $\Omega$ with the maximum element diameter $1/h=2$ in 2D or $1/h=3$ in 3D. We refine the mesh uniformly by connecting the midpoints of the element boundaries in two dimensions, whereas in three dimensions, we apply one-step bisection refinement and split each coarse-grid tetrahedron into two due to computational capacity limitations. Vertex-patched block smoothing and relaxation method are used when $d=2$, while edge-patched block smoothing and relaxation method are used when $d=3$ to save memory usage.

\subsubsection{Generalized Stokes equation}
\label{subsec:stokesLidNum}
We conducted tests on the PCG method preconditioned by both the variable V-cycle multigrid and W-cycle multigrid to solve the primal variable operator of the augmented Lagrangian Uzawa iteration for the condensed $H$(div) HDG scheme for the generalized Stokes equation. Table \ref{tab:stokesLid} reports the obtained PCG iteration counts with varying mesh levels, lower-order term coefficient $\beta$, finite element space polynomial order $k$, and smoothing steps $m_p = m_h = m$. The iteration counts remain almost unchanged as $\beta$ increases from 0 to $10^3$. 
We noticed a slight increase in iteration counts as the polynomial order $k$ increases, particularly in three-dimensional cases. However, increasing the smoothing steps effectively decreases the PCG iteration counts.
Other results verify the robustness of the $hp$-multigrid preconditioner with respect to mesh size.

\begin{table}
\caption{PCG iteration counts for the generalized Stokes equation in the lid-driven cavity problem in Example \ref{subsec:stokesLidNum}.}
\label{tab:stokesLid}
\begin{tabular}{|c|c|cc|cc|cc|cc|} 
\hline
\multirow{4}{*}{$k$} & \multirow{4}{*}{Level}                 & \multicolumn{4}{c|}{Variable V-cycle}                                         & \multicolumn{4}{c|}{W-cycle}                                                  \\ 
\cline{3-10}
                     &                                        & \multicolumn{2}{c|}{$\beta = 0$}      & \multicolumn{2}{c|}{$\beta = 1e{3}$}  & \multicolumn{2}{c|}{$\beta = 0$}      & \multicolumn{2}{c|}{$\beta = 1e{3}$}  \\ 
\cline{3-10}
                     &                                        & \multicolumn{8}{c|}{$d = 2$}                                                                                                                                  \\
                     &                                        & $m = 1$ & \multicolumn{1}{c}{$m = 2$} & $m = 1$ & \multicolumn{1}{c}{$m - 2$} & $m = 1$ & \multicolumn{1}{c}{$m = 2$} & $m = 1$ & $m = 2$                     \\ 
\hline
\multirow{4}{*}{0}   & 4                                      & 15      & 12                          & 10      & 7                           & 15      & 12                          & 10      & 7                           \\
                     & 5                                      & 18      & 13                          & 11      & 8                           & 15      & 11                          & 11      & 8                           \\
                     & 6                                      & 19      & 14                          & 14      & 10                          & 16      & 11                          & 13      & 10                          \\
                     & 7                                      & 20      & 14                          & 17      & 12                          & 16      & 11                          & 14      & 11                          \\ 
\hline
\multirow{4}{*}{1}   & 4                                      & 17      & 16                          & 10      & 9                           & 17      & 16                          & 10      & 9                           \\
                     & 5                                      & 19      & 16                          & 13      & 11                          & 18      & 14                          & 13      & 11                          \\
                     & 6                                      & 20      & 16                          & 18      & 13                          & 19      & 14                          & 15      & 13                          \\
                     & 7                                      & 20      & 16                          & 19      & 15                          & 19      & 14                          & 16      & 14                          \\ 
\hline
\multirow{4}{*}{2}   & 4                                      & 17      & 16                          & 10      & 10                          & 17      & 16                          & 10      & 10                          \\
                     & 5                                      & 19      & 16                          & 14      & 11                          & 18      & 14                          & 13      & 11                          \\
                     & 6                                      & 20      & 16                          & 16      & 13                          & 20      & 14                          & 15      & 13                          \\
                     & 7                                      & 20      & 16                          & 18      & 15                          & 20      & 14                          & 16      & 13                          \\ 
\hline
\multirow{4}{*}{3}   & 4                                      & 17      & 16                          & 10      & 10                          & 17      & 15                          & 10      & 10                          \\
                     & 5                                      & 19      & 16                          & 13      & 11                          & 19      & 14                          & 13      & 11                          \\
                     & 6                                      & 20      & 16                          & 16      & 13                          & 20      & 14                          & 15      & 13                          \\
                     & 7                                      & 21      & 16                          & 18      & 15                          & 20      & 14                          & 16      & 13                          \\ 
\hline
\multirow{2}{*}{ }   & \multicolumn{1}{l|}{\multirow{2}{*}{}} & \multicolumn{8}{c|}{$d = 3$}                                                                                                                                  \\
                     & \multicolumn{1}{l|}{}                  & $m = 2$ & \multicolumn{1}{c}{$m = 4$} & $m = 2$ & \multicolumn{1}{c}{$m = 4$} & $m = 2$ & \multicolumn{1}{c}{$m = 4$} & $m = 2$ & $m = 4$                     \\ 
\hline
\multirow{4}{*}{0}   & 7                                      & 10      & 9                           & 6       & 5                           & 9       & 8                           & 6       & 5                           \\
                     & 8                                      & 10      & 8                           & 8       & 6                           & 8       & 7                           & 8       & 6                           \\
                     & 9                                      & 9       & 8                           & 6       & 6                           & 8       & 7                           & 6       & 5                           \\
                     & 10                                     & 11      & 10                          & 7       & 7                           & 9       & 9                           & 7       & 6                           \\ 
\hline
\multirow{4}{*}{1}   & 7                                      & 14      & 12                          & 9       & 7                           & 14      & 12                          & 9       & 7                           \\
                     & 8                                      & 15      & 12                          & 9       & 8                           & 14      & 12                          & 9       & 8                           \\
                     & 9                                      & 15      & 13                          & 9       & 8                           & 15      & 13                          & 9       & 8                           \\
                     & 10                                     & 15      & 13                          & 10      & 9                           & 14      & 12                          & 10      & 9                           \\ 
\hline
\multirow{4}{*}{2}   & 7                                      & 17      & 13                          & 12      & 8                           & 17      & 13                          & 12      & 8                           \\
                     & 8                                      & 19      & 14                          & 11      & 9                           & 18      & 14                          & 11      & 9                           \\
                     & 9                                      & 18      & 15                          & 11      & 9                           & 18      & 15                          & 11      & 9                           \\
                     & 10                                     & 18      & 14                          & 13      & 10                          & 17      & 13                          & 13      & 10                          \\ 
\hline
\multirow{4}{*}{3}   & 7                                      & 21      & 15                          & 15      & 9                           & 21      & 15                          & 15      & 9                           \\
                     & 8                                      & 22      & 16                          & 15      & 9                           & 22      & 16                          & 15      & 9                           \\
                     & 9                                      & 22      & 16                          & 14      & 9                           & 22      & 16                          & 14      & 10                          \\
                     & 10                                     & 21      & 16                          & 17      & 11                          & 21      & 15                          & 17      & 11                          \\
\hline
\end{tabular}
\end{table}

\subsubsection{Stationary Navier-Stokes equations}
\label{subsec:nsLidNum}
We use PGM method preconditioned by the variable V-cycle method to solve the primal variable operator of the augmented Lagrangian Uzawa iteration method for the condensed $H$(div)-HDG method for the linearized Navier-Stokes equations as in Algorithm \ref{alg:ns}. Table \ref{tab:nsLid2D} and Table \ref{tab:nsLid3D} report the obtained average PGM iteration counts during the Picard iteration and Newton iteration procedures in Algorithm \ref{alg:st-hpMG}, with different mesh levels, viscosity $\nu$, finite element space polynomial order $k$, and smoothing steps $m_p = m_h = m$. Fig. \ref{fig:nsCavity} demonstrates the magnitude and the streamlines of the obtained numerical velocity solution with $\nu=10^{-3}$ (left panel) and $\nu=10^{-4}$ (right panel).

\begin{figure}
     \centering
     \begin{subfigure}[b]{0.45\textwidth}
         \centering
         \includegraphics[width=\textwidth]{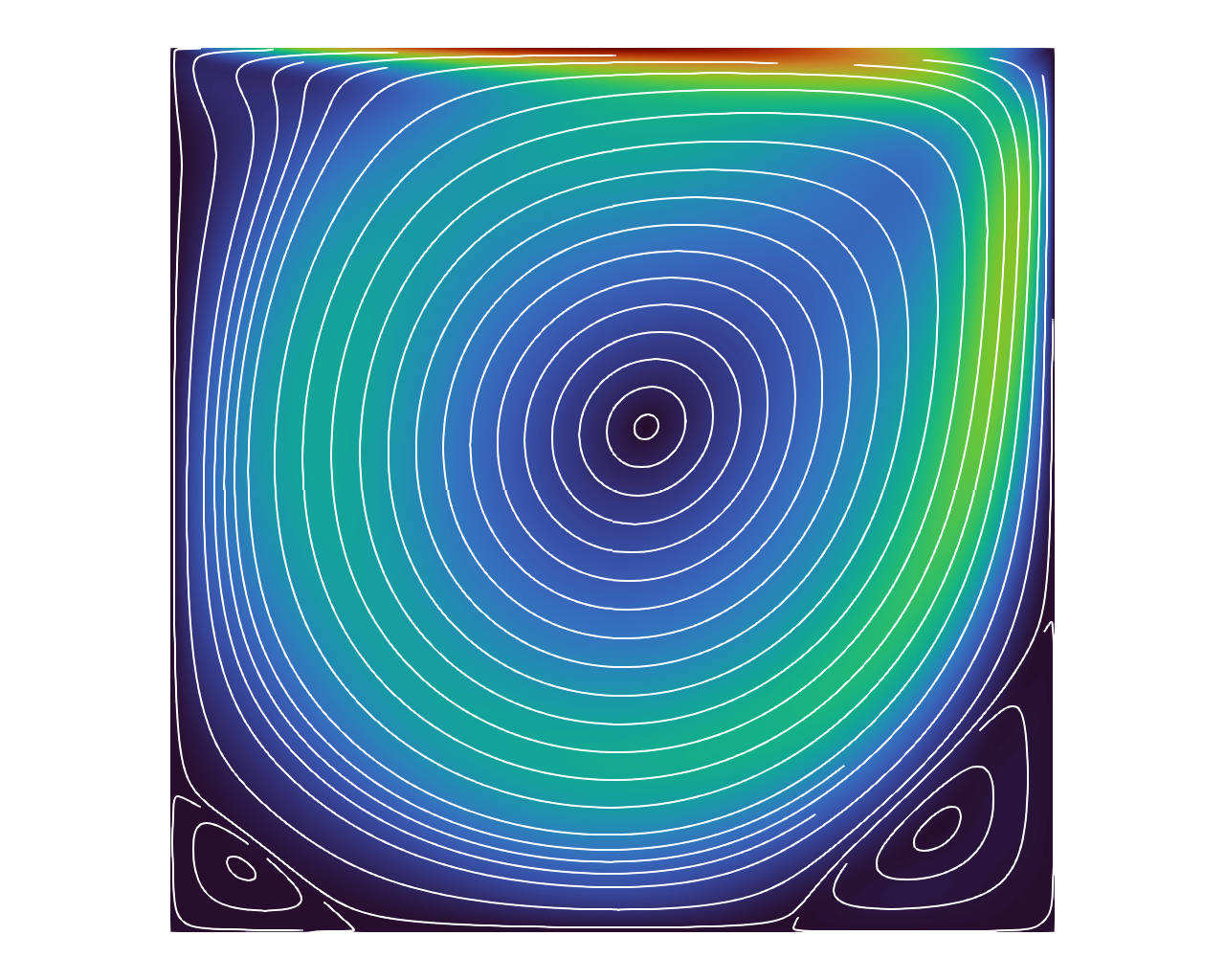}
         \caption{$\nu = 10^{-3}$}
     \end{subfigure}
     \hfill
     \begin{subfigure}[b]{0.45\textwidth}
         \centering
         \includegraphics[width=\textwidth]{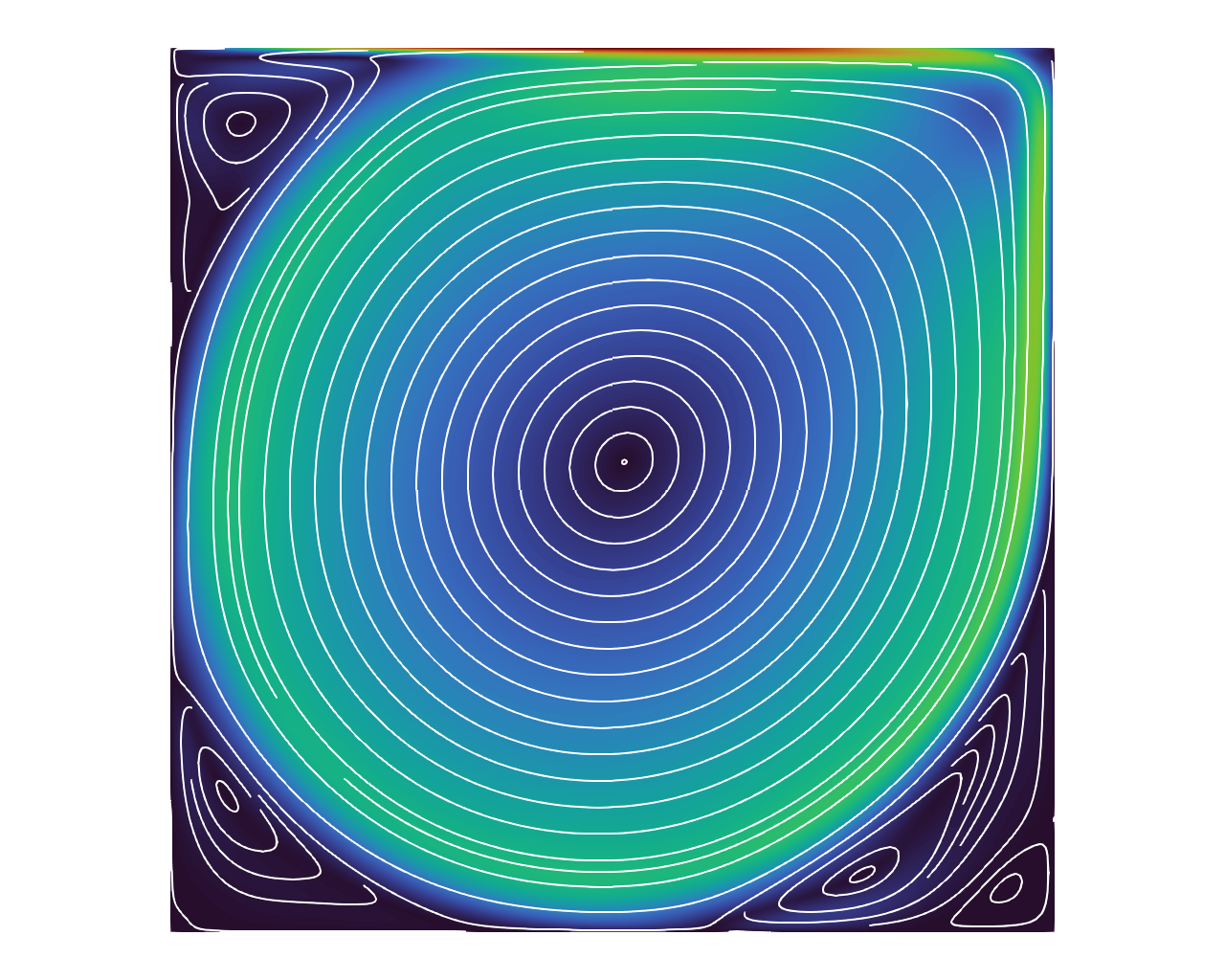}
         \caption{$\nu = 10^{-4}$}
     \end{subfigure}
        \caption{The obtained numerical velocity solution of the lid-driven cavity problem for the Navier-Stokes equations in Example \ref{subsec:nsLidNum}.}
        \label{fig:nsCavity}
\end{figure}

In Table \ref{tab:nsLid2D}, for the two-dimensional cases where $\nu\geq 10^{-3}$, the proposed $hp$-multigrid method demonstrates satisfactory performance under both Picard and Newton linearization methods, despite some mild increases in the iteration count of the PGM solver with increasing polynomial order $k$ and Reynolds number $\mathit{Re}$. However, when $\mathit{Re}=10^4$ and $k \geq 1$, the performance of the $hp$-multigrid method deteriorates significantly, although it remains satisfactory when $k=0$. This highlights the need for an efficient $p$-robust relaxation method to handle the convection-dominated case in the $H$(div)-HDG scheme, which requires further investigation.
In Table \ref{tab:nsLid3D}, the $hp$-multigrid method exhibits similarly good performance when $\nu\geq 10^{-2}$ for the three-dimensional cases. However, when $\nu\leq 10^{-3}$, the solution of the Picard linearization oscillates and cannot provide a sufficiently close initial guess for the Newton iteration. To address this issue, we employed the pseudo-time integration method to obtain the steady state, using a simple implicit backward-Euler discretization with the pseudo-time step size $\delta t^\ast$. The results in Table \ref{tab:nsLid3D} show that $1/\delta t^\ast=0.1$ is good enough when $\nu=10^{-3}$, and the extra mass term in the primal operator $\underline{A}_{k,h}^{c,\epsilon}$ in \eqref{nsOpEq-ALU1} makes it easier to be solved by PGM. In the extreme case where $\mathit{Re}=10^4$, the solution to the Picard iteration oscillates even more severely, and we omit this case to keep our discussion simple. All other results verify the robustness of our $hp$-multigrid with respect to the mesh size.

In addition to Algorithm \ref{alg:ns}, we are aware of other techniques to solve the stationary Navier-Stokes equations, such as continuation on Reynolds number and pseudo-time integration with implicit-explicit (IMEX) methods. Our numerical experiments demonstrate that the proposed $hp$-multigrid method is a promising preconditioner for these techniques.

\begin{table}
\caption{Two-dimensional average PGM iteration counts for the linearized Navier-Stokes equations during the Picard/Newton iteration procedure in Algorithm \ref{alg:ns} in the lid-driven cavity problem in Example \ref{subsec:nsLidNum}.}
\label{tab:nsLid2D}
\setlength{\tabcolsep}{18pt}
\begin{tabular}{|c|c|cccc|} 
\hline
\multicolumn{6}{|c|}{$d = 2$}                                                                                 \\ 
\hline
\multirow{3}{*}{$k$} & \multirow{3}{*}{Level} & \multicolumn{4}{c|}{$\nu$}                                    \\
                     &                        & $1$         & $1e{-2}$      & $1e{-3}$      & $1e{-4}$              \\ 
\cline{3-6}
                     &                        & \multicolumn{4}{c|}{Average PGM Iterations in Picard/Newton}  \\ 
\hline
\multicolumn{6}{|c|}{$m = 1$}                                                                                 \\ 
\hline
\multirow{4}{*}{0}   & 4                      & 12.0 / 6.0  & 14.4 / 8.0  & 18.2 / 11.0 & 25.6 / 16.8         \\
                     & 5                      & 12.5 / 5.5  & 16.5 / 8.2  & 23.1 / 11.7 & 36.8 / 24.2         \\
                     & 6                      & 12.5 / 5.5  & 17.7 / 9.3  & 27.1 / 16.2 & 45.2 / 29.0         \\
                     & 7                      & 12.0 / 5.5  & 17.7 / 9.7  & 31.6 / 21.2 & 54.2 / 36.5         \\ 
\hline
\multirow{4}{*}{1}   & 4                      & 18.0 / 8.5  & 22.3 / 12.3 & 40.3 / 28.5 & 68.0 / 50.9         \\
                     & 5                      & 20.0 / 9.0  & 25.7 / 13.3 & 47.5 / 31.8 & 116.2 / 91.7        \\
                     & 6                      & 20.0 / 8.0  & 27.3 / 15.0 & 52.5 / 37.0 & 150.1 / 134.0       \\
                     & 7                      & 19.0 / 8.0  & 27.0 / 13.5 & 54.1 / 38.0 & 166.2 / 151.0       \\ 
\hline
\multirow{4}{*}{2}   & 4                      & 21.0 / 9.5  & 25.3 / 15.0 & 44.9 / 33.3 & 92.9 / 86.0         \\
                     & 5                      & 23.0 /9.5   & 28.7 / 16.0 & 50.3 / 34.7 & 159.7 / 144.8       \\
                     & 6                      & 23.0 / 9.0  & 30.8 / 16.0 & 55.5 / 38.3 & 179.1 / 163.8       \\
                     & 7                      & 22.0 / 8.5  & 30.7 / 16.0 & 55.9 / 38.3 & 178.5 / 158.0       \\ 
\hline
\multirow{4}{*}{3}   & 4                      & 21.5 / 10.0 & 25.7 / 14.5 & 44.9 / 32.8 & 106.4 / 104.7       \\
                     & 5                      & 23.5 / 10.5 & 29.0 / 16.5 & 50.4 / 35.3 & 169.6 / 144.1       \\
                     & 6                      & 23.5 / 9.5  & 31.3 / 16.5 & 55.9 / 38.7 & 180.8 / 164.0       \\
                     & 7                      & 22.5 / 9.0  & 31.5 / 16.0 & 56.5 / 39.7 & 178.9 / 158.8       \\ 
\hline
\multicolumn{6}{|c|}{$m = 2$}                                                                                 \\ 
\hline
\multirow{4}{*}{0}   & 4                      & 9.5 / 4.5   & 11.4 / 6.5  & 13.8 / 8.5  & 17.6 / 11.2         \\
                     & 5                      & 10.0 / 4.5  & 13.0 / 6.8  & 18.0 / 11.2 & 24.0 / 15.6         \\
                     & 6                      & 10.0 / 4.5  & 14.0 / 8.0  & 22.5 / 14.4 & 29.5 / 20.3         \\
                     & 7                     & 9.0 / 4.5   & 14.2 / 8.3  & 26.7 / 18.0 & 38.4 / 27.5         \\ 
\hline
\multirow{4}{*}{1}   & 4                      & 13.5 / 6.0  & 17.8 / 10.0 & 31.0 / 22.5 & 46.5 / 35.9         \\
                     & 5                      & 14.0 / 6.0  & 20.0 / 10.3 & 38.9 / 26.8 & 85.7 / 69.0         \\
                     & 6                      & 13.0 / 5.5  & 19.7 / 11.0 & 43.7 / 30.7 & 119.8 / 105.2       \\
                     & 7                     & 12.5 / 5.5  & 18.7 / 10.5 & 43.1 / 30.0 & 136.0 / 123.2       \\ 
\hline
\multirow{4}{*}{2}   & 4                      & 14.5 / 6.5  & 18.7 / 11.5 & 32.7 / 24.7 & 56.9 / 52.2         \\
                     & 5                      & 15.0 / 6.5  & 20.7 / 11.5 & 40.0 / 27.7 & 107.8 / 98.2        \\
                     & 6                      & 14.0 / 6.5  & 20.5 / 11.5 & 44.5 / 30.0 & 136.1 / 124.0       \\
                     & 7                     & 13.5 / 6.0  & 19.5 / 10.5 & 43.8 / 30.7 & 141.4 / 125.8       \\ 
\hline
\multirow{4}{*}{3}   & 4                      & 14.5 / 6.5  & 18.7 / 11.5 & 32.9 / 24.8 & 62.3 / 62.0         \\
                     & 5                      & 15.0 / 6.5  & 20.7 / 12.0 & 40.0 / 28.7 & 114.5 / 96.7        \\
                     & 6                      & 14.5 / 6.5  & 20.7 / 11.5 & 44.5 / 31.3 & 136.9 / 123.7       \\
                     & 7                     & 13.5 / 6.0  & 19.7 / 11.0 & 43.9 / 31.0 & 141.4 / 126.0       \\
\hline
\end{tabular}
\end{table}

\begin{table}
\caption{Three-dimensional average PGM iteration counts for the linearized Navier-Stokes equations during the Picard/Newton iteration procedure in Algorithm \ref{alg:ns} in the lid-driven cavity problem in Example \ref{subsec:nsLidNum}.}
\label{tab:nsLid3D}
\setlength{\tabcolsep}{18pt}
\begin{tabular}{|ccccc|}
\hline
\multicolumn{5}{|c|}{$d = 3$}                                                                                                                                                                  \\ \hline
\multicolumn{1}{|c|}{\multirow{3}{*}{$k$}} & \multicolumn{1}{c|}{\multirow{3}{*}{Level}} & \multicolumn{3}{c|}{$\nu$}                                                                          \\
\multicolumn{1}{|c|}{}                     & \multicolumn{1}{c|}{}                       & $1$         & $1e{-2}$    & \begin{tabular}[c]{@{}c@{}}$1e{-3}$\\ (with $1/\delta t^\ast=0.1$)\end{tabular} \\ \cline{3-5} 
\multicolumn{1}{|c|}{}                     & \multicolumn{1}{c|}{}                       & \multicolumn{3}{c|}{Average PGM Iterations in Picard/Newton}                                        \\ \hline
\multicolumn{5}{|c|}{$m = 2$}                                                                                                                                                                  \\ \hline
\multicolumn{1}{|c|}{\multirow{4}{*}{0}}   & \multicolumn{1}{c|}{5}                      & 6.5 / 3.5   & 8.3 / 5.2   & 9.0 / 7.1                                                               \\
\multicolumn{1}{|c|}{}                     & \multicolumn{1}{c|}{6}                      & 6.5 / 3.5   & 7.6 / 5.0   & 9.2 / 6.9                                                               \\
\multicolumn{1}{|c|}{}                     & \multicolumn{1}{c|}{7}                      & 7.0 / 4.5   & 8.3 / 5.3   & 8.7 / 6.9                                                               \\
\multicolumn{1}{|c|}{}                     & \multicolumn{1}{c|}{8}                      & 8.0 / 4.5   & 11.3 / 7.4  & 12.8 / 10.9                                                             \\ \hline
\multicolumn{1}{|c|}{\multirow{4}{*}{1}}   & \multicolumn{1}{c|}{5}                      & 10.5 / 6.0  & 13.4 / 9.0  & 11.8 / 11.0                                                             \\
\multicolumn{1}{|c|}{}                     & \multicolumn{1}{c|}{6}                      & 13.5 / 7.0  & 16.4 / 10.8 & 13.5 / 12.8                                                             \\
\multicolumn{1}{|c|}{}                     & \multicolumn{1}{c|}{7}                      & 15.0 / 7.5  & 20.1 / 13.5 & 19.6 / 16.9                                                             \\
\multicolumn{1}{|c|}{}                     & \multicolumn{1}{c|}{8}                      & 14.5 / 7.5  & 18.6 / 11.8 & 19.9 / 17.0                                                             \\ \hline
\multicolumn{1}{|c|}{\multirow{4}{*}{2}}   & \multicolumn{1}{c|}{5}                      & 14.0 / 7.5  & 17.1 / 10.8 & 14.4 / 13.6                                                             \\
\multicolumn{1}{|c|}{}                     & \multicolumn{1}{c|}{6}                      & 17.0 / 9.0  & 22.0 / 13.5 & 17.7 / 16.0                                                             \\
\multicolumn{1}{|c|}{}                     & \multicolumn{1}{c|}{7}                      & 20.0 / 10.5 & 26.2 / 15.5 & 24.4 / 21.7                                                             \\
\multicolumn{1}{|c|}{}                     & \multicolumn{1}{c|}{8}                      & 19.5 / 1.0  & 25.2 / 16.0 & 24.4/ 20.4                                                              \\ \hline
\multicolumn{1}{|c|}{\multirow{4}{*}{3}}   & \multicolumn{1}{c|}{5}                      & 16.0 / 9.0  & 19.4 / 13.0 & 15.2 / 15.9                                                             \\
\multicolumn{1}{|c|}{}                     & \multicolumn{1}{c|}{6}                      & 19.5 / 11.5 & 24.9 / 15.8 & 19.9 / 19.3                                                             \\
\multicolumn{1}{|c|}{}                     & \multicolumn{1}{c|}{7}                      & 22.5 / 12.0 & 29.5 / 17.8 & 25.0 / 22.8                                                             \\
\multicolumn{1}{|c|}{}                     & \multicolumn{1}{c|}{8}                      & 23.0 / 12.0 & 29.1 / 17.3 & 25.5 / 22.9                                                             \\ \hline
\multicolumn{5}{|c|}{$m = 4$}                                                                                                                                                                  \\ \hline
\multicolumn{1}{|c|}{\multirow{4}{*}{0}}   & \multicolumn{1}{c|}{5}                      & 5.5 / 3.0   & 6.7 / 4.2   & 6.8 / 5.4                                                               \\
\multicolumn{1}{|c|}{}                     & \multicolumn{1}{c|}{6}                      & 5.5 / 3.0   & 6.6 / 4.2   & 6.9 / 5.8                                                               \\
\multicolumn{1}{|c|}{}                     & \multicolumn{1}{c|}{7}                      & 6.5 / 3.5   & 7.3 / 4.5   & 7.2 / 6.0                                                               \\
\multicolumn{1}{|c|}{}                     & \multicolumn{1}{c|}{8}                      & 7.0 / 3.5   & 9.2 / 5.8   & 9.8 / 8.1                                                               \\ \hline
\multicolumn{1}{|c|}{\multirow{4}{*}{1}}   & \multicolumn{1}{c|}{5}                      & 7.5 / 4.5   & 9.7 / 6.4   & 8.7 / 8.0                                                               \\
\multicolumn{1}{|c|}{}                     & \multicolumn{1}{c|}{6}                      & 10.0 / 5.5  & 11.6 / 8.0  & 10.4 / 9.8                                                              \\
\multicolumn{1}{|c|}{}                     & \multicolumn{1}{c|}{7}                      & 11.5 / 6.0  & 14.5 / 9.5  & 14.4 / 12.7                                                             \\
\multicolumn{1}{|c|}{}                     & \multicolumn{1}{c|}{8}                      & 10.5 / 5.5  & 13.5 / 8.5  & 15.1 / 2.0                                                              \\ \hline
\multicolumn{1}{|c|}{\multirow{4}{*}{2}}   & \multicolumn{1}{c|}{5}                      & 9.0 / 5.0   & 11.6 / 7.0  & 10.0 / 9.7                                                              \\
\multicolumn{1}{|c|}{}                     & \multicolumn{1}{c|}{6}                      & 12.5 / 6.5  & 14.5 / 9.5  & 12.6 / 11.4                                                             \\
\multicolumn{1}{|c|}{}                     & \multicolumn{1}{c|}{7}                      & 14.0 / 7.5  & 18.5 / 10.5 & 17.3 / 14.8                                                             \\
\multicolumn{1}{|c|}{}                     & \multicolumn{1}{c|}{8}                      & 13.0 / 7.0  & 16.8 / 10.3 & 17.2 / 15.2                                                             \\ \hline
\multicolumn{1}{|c|}{\multirow{4}{*}{3}}   & \multicolumn{1}{c|}{5}                      & 10.5 / 6.0  & 12.6 / 8.5  & 10.8 / 10.9                                                             \\
\multicolumn{1}{|c|}{}                     & \multicolumn{1}{c|}{6}                      & 13.5 / 7.5  & 16.4 / 11.0 & 13.6 / 13.9                                                             \\
\multicolumn{1}{|c|}{}                     & \multicolumn{1}{c|}{7}                      & 15.5 / 8.5  & 20.2 / 12.0 & 17.4 / 15.7                                                             \\
\multicolumn{1}{|c|}{}                     & \multicolumn{1}{c|}{8}                      & 15.5 / 8.5  & 18.6 / 11.3 & 17.7 / 16.2                                                             \\ \hline
\end{tabular}
\end{table}

\subsection{Backward-facing step flow problem}
Finally, we test the proposed $hp$-multigrid preconditioners for the 
backward-facing step flow problem, with the domain $\Omega=([0.5,4]\times[0,0.5])\cup([0,4]\times[0.5,1])$ when $d=2$, or $\Omega=(([0.5,4]\times[0,0.5])\cup([0,4]\times[0.5,1]))\times[0,1]$ when $d=3$. We assume an inhomogeneous Dirichlet boundary condition $\underline{u}=[16(1-y)(y-0.5),\; 0]^\trans$ when $d=2$, or $\underline{u}=[64(1-y)(y-0.5)z(1-z),\; 0,\; 0]^\trans$ when $d=3$ for the inlet flow on $\{x=0\}$, with do-nothing boundary condition on $\{x=4\}$ and no-slip boundary conditions on the remaining sides.
We note that when the edge-patched block relaxation method is used in our $hp$-multigrid, an evident increase in the iteration count of preconditioned Krylov subspace solvers with the increase of the polynomial degree is observed when the length of the domain is increased. Thus in the numerical experiments for the backward-facing step flow problems, we use vertex-patched relaxation and edge-patched smoothing in the $hp$-multigrid.
Other settings are the same as in the lid-driven cavity problem. 

\subsubsection{Generalized Stokes equation}
\label{subsec:stokesBackNum}
Table \ref{tab:stokesBack} reports the obtained PCG iteration counts.
Similar optimal results are observed as in the lid-driven cavity problem.
We note that when $d=2$, $\beta=10^3$ and $m=1$, or when $d=3$ and $m=2$, the PCG solver preconditioned by the W-cycle multigrid fails, which is due to the fact that the smoothing steps on each mesh level of the W-cycle multigrid are not large enough to make the preconditioned operator definite and robust with respect to mesh size \cite[Section 4]{BPX1991}. Meanwhile, the variable V-cycle with $m=1$ remains a robust and positive definite preconditioner.

\begin{table}
\caption{PCG iteration counts for the generalized Stokes equation in the backward-facing step flow problem in Example \ref{subsec:stokesBackNum}, where "NA" means the conjugate gradient solver fails.}
\label{tab:stokesBack}
\begin{tabular}{|c|c|cc|cc|cc|cc|} 
\hline
\multirow{4}{*}{$k$} & \multirow{4}{*}{Level}                 & \multicolumn{4}{c|}{Variable V-cycle}                                          & \multicolumn{4}{c|}{W-cycle}                                                  \\ 
\cline{3-10}
                     &                                        & \multicolumn{2}{c|}{$\beta = 0$}       & \multicolumn{2}{c|}{$\beta = 1e{3}$}  & \multicolumn{2}{c|}{$\beta = 0$}      & \multicolumn{2}{c|}{$\beta = 1e{3}$}  \\ 
\cline{3-10}
                     &                                        & \multicolumn{8}{c|}{$d = 2$}                                                                                                                                   \\
                     &                                        & $m  = 1$ & \multicolumn{1}{c}{$m = 2$} & $m = 1$ & \multicolumn{1}{c}{$m = 2$} & $m = 1$ & \multicolumn{1}{c}{$m = 2$} & $m = 1$ & $m = 2$                     \\ 
\hline
\multirow{4}{*}{0}   & 4                                      & 13       & 9                           & 9       & 6                           & 10      & 8                           & NA      & 6                           \\
                     & 5                                      & 14       & 10                          & 10      & 8                           & 11      & 7                           & NA      & 7                           \\
                     & 6                                      & 14       & 10                          & 12      & 9                           & 13      & 8                           & NA      & 8                           \\
                     & 7                                      & 15       & 10                          & 14      & 10                          & 12      & 8                           & NA      & 7                           \\ 
\hline
\multirow{4}{*}{1}   & 4                                      & 14       & 12                          & 10      & 8                           & 14      & 11                          & NA      & 8                           \\
                     & 5                                      & 15       & 13                          & 13      & 10                          & 15      & 10                          & NA      & 10                          \\
                     & 6                                      & 16       & 13                          & 15      & 12                          & 16      & 11                          & NA      & 11                          \\
                     & 7                                      & 17       & 13                          & 17      & 13                          & 16      & 10                          & NA      & 11                          \\ 
\hline
\multirow{4}{*}{2}   & 4                                      & 15       & 12                          & 10      & 9                           & 14      & 11                          & NA      & 8                           \\
                     & 5                                      & 16       & 13                          & 13      & 11                          & 15      & 11                          & NA      & 10                          \\
                     & 6                                      & 17       & 13                          & 15      & 12                          & 17      & 11                          & NA      & 11                          \\
                     & 7                                      & 17       & 13                          & 17      & 13                          & 16      & 11                          & NA      & 11                          \\ 
\hline
\multirow{4}{*}{3}   & 4                                      & 15       & 12                          & 10      & 9                           & 14      & 11                          & NA      & 8                           \\
                     & 5                                      & 16       & 13                          & 13      & 11                          & 15      & 11                          & NA      & 10                          \\
                     & 6                                      & 17       & 13                          & 15      & 12                          & 17      & 11                          & NA      & 11                          \\
                     & 7                                      & 17       & 13                          & 17      & 13                          & 16      & 11                          & NA      & 11                          \\ 
\hline
\multirow{2}{*}{}    & \multicolumn{1}{l|}{\multirow{2}{*}{}} & \multicolumn{8}{c|}{$d = 3$}                                                                                                                                   \\
                     & \multicolumn{1}{l|}{}                  & $m = 2$  & \multicolumn{1}{c}{$m = 4$} & $m = 2$ & \multicolumn{1}{c}{$m = 4$} & $m = 2$ & \multicolumn{1}{c}{$m = 4$} & $m = 2$ & $m = 4$                     \\ 
\hline
\multirow{4}{*}{0}   & 5                                      & 11       & 9                           & 13      & 8                           & 21      & 9                           & NA      & 8                           \\
                     & 6                                      & 12       & 9                           & 11      & 7                           & 28      & 9                           & NA      & 7                           \\
                     & 7                                      & 12       & 9                           & 10      & 7                           & NA      & 10                          & NA      & 7                           \\
                     & 8                                      & 12       & 9                           & 10      & 7                           & NA      & 8                           & NA      & 7                           \\ 
\hline
\multirow{4}{*}{1}   & 5                                      & 13       & 10                          & 7       & 6                           & 12      & 8                           & 49      & 6                           \\
                     & 6                                      & 14       & 11                          & 8       & 7                           & 18      & 10                          & NA      & 7                           \\
                     & 7                                      & 15       & 12                          & 9       & 7                           & NA      & 9                           & NA      & 7                           \\
                     & 8                                      & 16       & 13                          & 10      & 8                           & NA      & 11                          & NA      & 8                           \\ 
\hline
\multirow{4}{*}{2}   & 5                                      & 14       & 11                          & 8       & 6                           & 14      & 10                          & NA      & 6                           \\
                     & 6                                      & 15       & 12                          & 9       & 7                           & 21      & 11                          & NA      & 7                           \\
                     & 7                                      & 16       & 13                          & 10      & 8                           & NA      & 11                          & NA      & 7                           \\
                     & 8                                      & 18       & 14                          & 11      & 9                           & NA      & 12                          & NA      & 8                           \\ 
\hline
\multirow{4}{*}{3}   & 4                                      & 13       & 11                          & 7       & 6                           & 12      & 10                          & 30      & 6                           \\
                     & 5                                      & 14       & 11                          & 8       & 6                           & 14      & 10                          & NA      & 6                           \\
                     & 6                                      & 16       & 12                          & 9       & 7                           & 21      & 11                          & NA      & 7                           \\
                     & 7                                      & 17       & 13                          & 10      & 8                           & NA      & 11                          & NA      & 7                           \\
\hline
\end{tabular}
\end{table}

\subsubsection{Stationary Navier-Stokes equations}
\label{subsec:nsBackNum}
We apply the PGM method preconditioned by the variable V-cycle method to solve the augmented Lagrangian Uzawa iteration \eqref{nsOpEq-ALU1} for the condensed $H$(div)-HDG method for the linearized Navier-Stokes equations. Here we limit ourselves to cases with two-dimensional domain and with Reynolds number up to $10^3$. All other settings are identical to those in Example \ref{subsec:stokesBackNum} for the generalized Stokes equation. The resulting average PGM iteration counts during the Picard and Newton iteration procedures in Algorithm are reported in Table \ref{tab:nsBack2D} for different mesh levels, viscosity $\nu$, finite element space polynomial order $k$, and smoothing steps $m_p = m_h = m$, and similar results as for the lid-driven cavity problem in Example \ref{subsec:nsLidNum} are observed.

\begin{table}
\caption{Two-dimensional average PGM iteration counts for the linearized Navier-Stokes equations during the Picard/Newton iteration procedure in Algorithm \ref{alg:ns} in the backward-facing step flow problem in Example \ref{subsec:nsBackNum}.}
\label{tab:nsBack2D}
\setlength{\tabcolsep}{18pt}
\begin{tabular}{|ccccc|}
\hline
\multicolumn{5}{|c|}{$d = 2$}                                                                                                                           \\ \hline
\multicolumn{1}{|c|}{\multirow{3}{*}{$k$}} & \multicolumn{1}{c|}{\multirow{3}{*}{Level}} & \multicolumn{3}{c|}{$\nu$}                                   \\
\multicolumn{1}{|c|}{}                     & \multicolumn{1}{c|}{}                       & $1$                & $1e{-2}$           & $1e{-3}$           \\ \cline{3-5} 
\multicolumn{1}{|c|}{}                     & \multicolumn{1}{c|}{}                       & \multicolumn{3}{c|}{Average PGM Iterations in Picard/Newton} \\ \hline
\multicolumn{5}{|c|}{$m = 1$}                                                                                                                           \\ \hline
\multicolumn{1}{|c|}{\multirow{4}{*}{0}}   & \multicolumn{1}{c|}{4}                      & 15.0 / 9.0         & 18.0 / 11.8        & 37.6 / 35.1        \\
\multicolumn{1}{|c|}{}                     & \multicolumn{1}{c|}{5}                      & 15.0 / 8.0         & 20.3 / 13.0        & 34.2 / 30.8        \\
\multicolumn{1}{|c|}{}                     & \multicolumn{1}{c|}{6}                      & 15.0 / 8.0         & 21.2 / 13.5        & 29.8 / 23.8        \\
\multicolumn{1}{|c|}{}                     & \multicolumn{1}{c|}{7}                      & 15.0 / 7.5         & 21.8 / 14.3        & 28.8 / 22.8        \\ \hline
\multicolumn{1}{|c|}{\multirow{4}{*}{1}}   & \multicolumn{1}{c|}{4}                      & 23.0 / 12.5        & 29.6 / 17.7        & 44.7 / 42.3        \\
\multicolumn{1}{|c|}{}                     & \multicolumn{1}{c|}{5}                      & 24.0 / 12.5        & 32.1 / 20.0        & 50.6 / 38.0        \\
\multicolumn{1}{|c|}{}                     & \multicolumn{1}{c|}{6}                      & 24.5 / 13.0        & 32.3 / 19.5        & 52.9 / 35.7        \\
\multicolumn{1}{|c|}{}                     & \multicolumn{1}{c|}{7}                      & 25.0 / 13.0        & 32.3 / 19.5        & 51.0 / 34.3        \\ \hline
\multicolumn{1}{|c|}{\multirow{4}{*}{2}}   & \multicolumn{1}{c|}{3}                      & 27.0 / 14.0        & 28.0 / 21.0        & 37.3 / 37.5        \\
\multicolumn{1}{|c|}{}                     & \multicolumn{1}{c|}{4}                      & 27.0 / 14.5        & 32.0 / 20.0        & 48.2 / 40.0        \\
\multicolumn{1}{|c|}{}                     & \multicolumn{1}{c|}{5}                      & 28.5 / 14.5        & 35.7 / 20.5        & 55.0 / 43.0        \\
\multicolumn{1}{|c|}{}                     & \multicolumn{1}{c|}{6}                      & 28.5 / 14.5        & 36.2 / 21.0        & 59.1 / 42.2        \\ \hline
\multicolumn{1}{|c|}{\multirow{4}{*}{3}}   & \multicolumn{1}{c|}{3}                      & 25.0 / 15.0        & 26.2 / 18.0        & 37.9 / 35.2        \\
\multicolumn{1}{|c|}{}                     & \multicolumn{1}{c|}{4}                      & 27.5 / 15.0        & 32.6 / 20.0        & 48.7 / 41.2        \\
\multicolumn{1}{|c|}{}                     & \multicolumn{1}{c|}{5}                      & 29.0 / 15.5        & 36.3 / 21.0        & 56.7 / 44.0        \\
\multicolumn{1}{|c|}{}                     & \multicolumn{1}{c|}{6}                      & 29.5 / 15.5        & 37.3 / 22.0        & 59.5 / 42.8        \\ \hline
\multicolumn{5}{|c|}{$m = 2$}                                                                                                                           \\ \hline
\multicolumn{1}{|c|}{\multirow{4}{*}{0}}   & \multicolumn{1}{c|}{4}                      & 11.0 / 6.5        & 14.1 / 9.8         & 19.4 / 19.6        \\
\multicolumn{1}{|c|}{}                     & \multicolumn{1}{c|}{5}                      & 11.0 / 6.5        & 16.1 / 10.5        & 18.9 / 17.7        \\
\multicolumn{1}{|c|}{}                     & \multicolumn{1}{c|}{6}                      & 11.0 / 6.5        & 17.0 / 10.5        & 21.0 / 17.2        \\
\multicolumn{1}{|c|}{}                     & \multicolumn{1}{c|}{7}                      & 11.0 / 6.0        & 17.2 / 11.3        & 22.6 / 17.8        \\ \hline
\multicolumn{1}{|c|}{\multirow{4}{*}{1}}   & \multicolumn{1}{c|}{4}                      & 15.5 / 9.0         & 23.9 / 14.7        & 30.3 / 28.0        \\
\multicolumn{1}{|c|}{}                     & \multicolumn{1}{c|}{5}                      & 15.5 / 9.0         & 24.6 / 15.5        & 36.4 / 27.0        \\
\multicolumn{1}{|c|}{}                     & \multicolumn{1}{c|}{6}                      & 16.5 / 9.0         & 23.8 / 15.0        & 38.6 / 26.7        \\
\multicolumn{1}{|c|}{}                     & \multicolumn{1}{c|}{7}                      & 16.0 / 8.5         & 22.9 / 14.5        & 37.7 / 25.7        \\ \hline
\multicolumn{1}{|c|}{\multirow{4}{*}{2}}   & \multicolumn{1}{c|}{3}                      & 16.0 / 9.5         & 19.8 / 14.5        & 21.9 / 22.3        \\
\multicolumn{1}{|c|}{}                     & \multicolumn{1}{c|}{4}                      & 16.5 / 9.0         & 24.8 / 16.0        & 29.8 / 25.7        \\
\multicolumn{1}{|c|}{}                     & \multicolumn{1}{c|}{5}                      & 17.5 / 10.0        & 25.2 / 16.0        & 36.7 / 27.7        \\
\multicolumn{1}{|c|}{}                     & \multicolumn{1}{c|}{6}                      & 18.0 / 10.0        & 24.0 / 15.5        & 38.8 / 27.0        \\ \hline
\multicolumn{1}{|c|}{\multirow{4}{*}{3}}   & \multicolumn{1}{c|}{3}                      & 16.0 / 9.5         & 19.8 / 14.0        & 22.4 / 21.7        \\
\multicolumn{1}{|c|}{}                     & \multicolumn{1}{c|}{4}                      & 16.5 / 10.0        & 24.8 / 16.0        & 30.0 / 26.3        \\
\multicolumn{1}{|c|}{}                     & \multicolumn{1}{c|}{5}                      & 17.5 / 10.0        & 25.4 / 16.0        & 36.3 / 28.7        \\
\multicolumn{1}{|c|}{}                     & \multicolumn{1}{c|}{6}                      & 18.0 / 10.0        & 24.2 / 15.5        & 39.1 / 26.7        \\ \hline
\end{tabular}
\end{table}

\section{Conclusion}
\label{sec:conclude}
In this study, we developed an $hp$-multigrid preconditioner for the $H$(div)-HDG scheme for both the generalized Stokes and the Navier-Stokes equations. The condensed $H$(div)-HDG system is solved by the augmented Lagrangian Uzawa iteration, and the $hp$-multigrid is used to precondition the nearly-singular primal operator on the global velocity spaces. The proposed $hp$-multigrid is essentially a multiplicative ASP, with the lowest global velocity space as the auxiliary space and a robust geometric multigrid algorithm as the auxiliary space solver. For the generalized Stokes equation, we prove that the condensed lowest-order $H$(div)-HDG discretization is equivalent to the CR discretization with a pressure-robust treatment, which allows for the application of the rich geometric multigrid theory for the CR discretization. Numerical experiments demonstrate that the proposed $hp$-multigrid is robust to mesh size and the augmented Lagrangian parameter, while a very mild increase in iteration counts of the preconditioned Krylov space solvers with respect to the increase of polynomial order is observed. We further test the proposed $hp$-multigrid preconditioner on the $H$(div)-HDG scheme for the linearized Naiver-Stokes equation by Picard or Newton's method, and the iteration counts grow mildly with respect to the increase of the Reynolds number as large as $10^3$. An efficient parallel implementation of an additive $hp$-multigrid algorithm and the proof of its robustness for the generalized Stokes equation will be our forthcoming research.

\bibliography{mgHdivHDG}

\begin{thebibliography}{10}

\bibitem{arnold1985mixed}
{\sc D.~N. Arnold and F.~Brezzi}, {\em Mixed and nonconforming finite element
  methods: implementation, postprocessing and error estimates}, RAIRO
  Mod\'{e}l. Math. Anal. Num\'{e}r., 19 (1985), pp.~7--32.

\bibitem{arnold2000multigrid}
{\sc D.~N. Arnold, R.~S. Falk, and R.~Winther}, {\em Multigrid in {$H({\rm
  div})$} and {$H({\rm curl})$}}, Numer. Math., 85 (2000), pp.~197--217.

\bibitem{ayachit2015paraview}
{\sc U.~Ayachit}, {\em The paraview guide: a parallel visualization
  application}, Kitware, Inc., 2015.

\bibitem{benziSaddle2005}
{\sc M.~Benzi, G.~H. Golub, and J.~Liesen}, {\em Numerical solution of saddle
  point problems}, Acta Numer., 14 (2005), pp.~1--137.

\bibitem{benziOlshanskii06}
{\sc M.~Benzi and M.~A. Olshanskii}, {\em An augmented {L}agrangian-based
  approach to the {O}seen problem}, SIAM J. Sci. Comput., 28 (2006),
  pp.~2095--2113.

\bibitem{braess1990multigrid}
{\sc D.~Braess and R.~Verf\"{u}rth}, {\em Multigrid methods for nonconforming
  finite element methods}, SIAM J. Numer. Anal., 27 (1990), pp.~979--986.

\bibitem{BPX1991}
{\sc J.~H. Bramble, J.~E. Pasciak, and J.~Xu}, {\em The analysis of multigrid
  algorithms with nonnested spaces or noninherited quadratic forms}, Math.
  Comp., 56 (1991), pp.~1--34.

\bibitem{brenner1989optimal}
{\sc S.~C. Brenner}, {\em An optimal-order multigrid method for {${\rm P}1$}
  nonconforming finite elements}, Math. Comp., 52 (1989), pp.~1--15.

\bibitem{brenner1990nonconforming}
\leavevmode\vrule height 2pt depth -1.6pt width 23pt, {\em A nonconforming
  multigrid method for the stationary {S}tokes equations}, Math. Comp., 55
  (1990), pp.~411--437.

\bibitem{brenner1993nonconforming}
\leavevmode\vrule height 2pt depth -1.6pt width 23pt, {\em A nonconforming
  mixed multigrid method for the pure displacement problem in planar linear
  elasticity}, SIAM J. Numer. Anal., 30 (1993), pp.~116--135.

\bibitem{brenner1994nonconforming}
\leavevmode\vrule height 2pt depth -1.6pt width 23pt, {\em A nonconforming
  mixed multigrid method for the pure traction problem in planar linear
  elasticity}, Math. Comp., 63 (1994), pp.~435--460, S1--S5.

\bibitem{Brubeck22}
{\sc P.~D. Brubeck and P.~E. Farrell}, {\em A scalable and robust vertex-star
  relaxation for high-order {FEM}}, SIAM J. Sci. Comput., 44 (2022),
  pp.~A2991--A3017.

\bibitem{cockburnHDGMG}
{\sc B.~Cockburn, O.~Dubois, J.~Gopalakrishnan, and S.~Tan}, {\em Multigrid for
  an {HDG} method}, IMA J. Numer. Anal., 34 (2014), pp.~1386--1425.

\bibitem{cockburnProjection}
{\sc B.~Cockburn, J.~Gopalakrishnan, and F.-J. Sayas}, {\em A projection-based
  error analysis of {HDG} methods}, Math. Comp., 79 (2010), pp.~1351--1367.

\bibitem{cockburnHdivHDGStokes}
{\sc B.~Cockburn and F.-J. Sayas}, {\em Divergence-conforming {HDG} methods for
  {S}tokes flows}, Math. Comp., 83 (2014), pp.~1571--1598.

\bibitem{CR1973}
{\sc M.~Crouzeix and P.-A. Raviart}, {\em Conforming and nonconforming finite
  element methods for solving the stationary {S}tokes equations. {I}}, Rev.
  Fran\c{c}aise Automat. Informat. Recherche Op\'{e}rationnelle S\'{e}r. Rouge,
  7 (1973), pp.~33--75.

\bibitem{Farrell20}
{\sc P.~E. Farrell, L.~Mitchell, L.~R. Scott, and F.~Wechsung}, {\em Robust
  multigrid methods for nearly incompressible elasticity using macro elements},
  IMA J. Numer. Anal., 42 (2022), pp.~3306--3329.

\bibitem{farrell2019augmented}
{\sc P.~E. Farrell, L.~Mitchell, and F.~Wechsung}, {\em An augmented
  {L}agrangian preconditioner for the 3{D} stationary incompressible
  {N}avier-{S}tokes equations at high {R}eynolds number}, SIAM J. Sci. Comput.,
  41 (2019), pp.~A3073--A3096.

\bibitem{fortin2000augmented}
{\sc M.~Fortin and R.~Glowinski}, {\em Augmented {L}agrangian methods}, vol.~15
  of Studies in Mathematics and its Applications, North-Holland Publishing Co.,
  Amsterdam, 1983.
\newblock Applications to the numerical solution of boundary value problems,
  Translated from the French by B. Hunt and D. C. Spicer.

\bibitem{fuQiuBrinkman}
{\sc G.~Fu, Y.~Jin, and W.~Qiu}, {\em Parameter-free superconvergent {$H({\rm
  div})$}-conforming {HDG} methods for the {B}rinkman equations}, IMA J. Numer.
  Anal., 39 (2019), pp.~957--982.

\bibitem{fk2022optimal}
{\sc G.~Fu and W.~Kuang}, {\em Optimal geometric multigrid preconditioners for
  hdg-p0 schemes for the reaction-diffusion equation and the generalized stokes
  equations}, arXiv preprint arXiv:2208.14418,  (2022).

\bibitem{FuKuangBlock}
\leavevmode\vrule height 2pt depth -1.6pt width 23pt, {\em {Uniform
  block-diagonal preconditioners for divergence-conforming HDG Methods for the
  generalized Stokes equations and the linear elasticity equations}}, IMA
  Journal of Numerical Analysis,  (2022).
\newblock drac021.

\bibitem{hong2016robust}
{\sc Q.~Hong, J.~Kraus, J.~Xu, and L.~Zikatanov}, {\em A robust multigrid
  method for discontinuous {G}alerkin discretizations of {S}tokes and linear
  elasticity equations}, Numer. Math., 132 (2016), pp.~23--49.

\bibitem{LinkeReview2017}
{\sc V.~John, A.~Linke, C.~Merdon, M.~Neilan, and L.~G. Rebholz}, {\em On the
  divergence constraint in mixed finite element methods for incompressible
  flows}, SIAM Rev., 59 (2017), pp.~492--544.

\bibitem{Kanschat15}
{\sc G.~Kanschat and Y.~Mao}, {\em Multigrid methods for {$H^{\rm
  div}$}-conforming discontinuous {G}alerkin methods for the {S}tokes
  equations}, J. Numer. Math., 23 (2015), pp.~51--66.

\bibitem{Lee09}
{\sc Y.-J. Lee, J.~Wu, and J.~Chen}, {\em Robust multigrid method for the
  planar linear elasticity problems}, Numer. Math., 113 (2009), pp.~473--496.

\bibitem{lee2007robust}
{\sc Y.-J. Lee, J.~Wu, J.~Xu, and L.~Zikatanov}, {\em Robust subspace
  correction methods for nearly singular systems}, Math. Models Methods Appl.
  Sci., 17 (2007), pp.~1937--1963.

\bibitem{lehrenfeld2010hybrid}
{\sc C.~Lehrenfeld}, {\em Hybrid {Discontinuous} {Galerkin} methods for solving
  incompressible flow problems}, PhD thesis, RWTH Aachen University, 2010.

\bibitem{lehrenfeldHighOrder}
{\sc C.~Lehrenfeld and J.~Sch\"{o}berl}, {\em High order exactly
  divergence-free hybrid discontinuous {G}alerkin methods for unsteady
  incompressible flows}, Comput. Methods Appl. Mech. Engrg., 307 (2016),
  pp.~339--361.

\bibitem{linkeRT2014}
{\sc A.~Linke}, {\em On the role of the {H}elmholtz decomposition in mixed
  methods for incompressible flows and a new variational crime}, Comput.
  Methods Appl. Mech. Engrg., 268 (2014), pp.~782--800.

\bibitem{linkNS2016}
{\sc A.~Linke and C.~Merdon}, {\em Pressure-robustness and discrete {H}elmholtz
  projectors in mixed finite element methods for the incompressible
  {N}avier-{S}tokes equations}, Comput. Methods Appl. Mech. Engrg., 311 (2016),
  pp.~304--326.

\bibitem{luHMGHDG}
{\sc P.~Lu, A.~Rupp, and G.~Kanschat}, {\em Homogeneous multigrid for {HDG}},
  IMA J. Numer. Anal., 42 (2022), pp.~3135--3153.

\bibitem{lu2023homogeneous}
{\sc P.~Lu, W.~Wang, G.~Kanschat, and A.~Rupp}, {\em Homogeneous multigrid
  method for hdg applied to the stokes equation}, arXiv preprint
  arXiv:2302.00121,  (2023).

\bibitem{marini1985inexpensive}
{\sc L.~D. Marini}, {\em An inexpensive method for the evaluation of the
  solution of the lowest order {R}aviart-{T}homas mixed method}, SIAM J. Numer.
  Anal., 22 (1985), pp.~493--496.

\bibitem{pavarino94}
{\sc L.~F. Pavarino}, {\em Additive {S}chwarz methods for the {$p$}-version
  finite element method}, Numer. Math., 66 (1994), pp.~493--515.

\bibitem{RTspace1977}
{\sc P.-A. Raviart and J.~M. Thomas}, {\em A mixed finite element method for
  2nd order elliptic problems}, in Mathematical aspects of finite element
  methods ({P}roc. {C}onf., {C}onsiglio {N}az. delle {R}icerche ({C}.{N}.{R}.),
  {R}ome, 1975), Lecture Notes in Math., Vol. 606, Springer, Berlin, 1977,
  pp.~292--315.

\bibitem{rhebergenStokes}
{\sc S.~Rhebergen and G.~N. Wells}, {\em Analysis of a hybridized/interface
  stabilized finite element method for the {S}tokes equations}, SIAM J. Numer.
  Anal., 55 (2017), pp.~1982--2003.

\bibitem{rhebergenNS}
\leavevmode\vrule height 2pt depth -1.6pt width 23pt, {\em A hybridizable
  discontinuous {G}alerkin method for the {N}avier-{S}tokes equations with
  pointwise divergence-free velocity field}, J. Sci. Comput., 76 (2018),
  pp.~1484--1501.

\bibitem{rhebergenPrecond1}
\leavevmode\vrule height 2pt depth -1.6pt width 23pt, {\em Preconditioning of a
  hybridized discontinuous {G}alerkin finite element method for the {S}tokes
  equations}, J. Sci. Comput., 77 (2018), pp.~1936--1952.

\bibitem{rhebergenPrecond2}
\leavevmode\vrule height 2pt depth -1.6pt width 23pt, {\em Preconditioning for
  a pressure-robust {HDG} discretization of the {S}tokes equations}, SIAM J.
  Sci. Comput., 44 (2022), pp.~A583--A604.

\bibitem{gmres}
{\sc Y.~Saad and M.~H. Schultz}, {\em G{MRES}: a generalized minimal residual
  algorithm for solving nonsymmetric linear systems}, SIAM J. Sci. Statist.
  Comput., 7 (1986), pp.~856--869.

\bibitem{schoberl1999multigrid}
{\sc J.~Sch{\"o}berl}, {\em Multigrid methods for a parameter dependent problem
  in primal variables}, Numer. Math., 84 (1999), pp.~97--119.

\bibitem{schoberl1999robust}
\leavevmode\vrule height 2pt depth -1.6pt width 23pt, {\em Robust multigrid
  methods for parameter dependent problems}, PhD thesis, Johannes Kepler
  University Linz, 1999.

\bibitem{Schoberl16}
\leavevmode\vrule height 2pt depth -1.6pt width 23pt, {\em {C}++11
  {I}mplementation of {F}inite {E}lements in {NGS}olve}, 2014.
\newblock {ASC Report 30/2014, Institute for Analysis and Scientific Computing,
  Vienna University of Technology}.

\bibitem{stevenson1998cascade}
{\sc R.~Stevenson}, {\em Nonconforming finite elements and the cascadic
  multi-grid method}, Numer. Math., 91 (2002), pp.~351--387.

\bibitem{Tan09}
{\sc S.~Tan}, {\em Iterative solvers for hybridized finite element methods},
  PhD thesis, University of Florida, 2009.

\bibitem{turek1994multigrid}
{\sc S.~Turek}, {\em Multigrid techniques for a divergence-free finite element
  discretization}, East-West J. Numer. Math., 2 (1994), pp.~229--255.

\bibitem{uzawa1958iterative}
{\sc H.~Uzawa}, {\em Iterative methods for concave programming}, in Studies in
  linear and non-linear programming, Stanford University Press, Stanford,
  Calif., 1958, pp.~154--165.

\bibitem{XuSubspace92}
{\sc J.~Xu}, {\em Iterative methods by space decomposition and subspace
  correction}, SIAM Rev., 34 (1992), pp.~581--613.

\bibitem{Xu96}
{\sc J.~Xu}, {\em The auxiliary space method and optimal multigrid
  preconditioning techniques for unstructured grids}, vol.~56, 1996,
  pp.~215--235.
\newblock International GAMM-Workshop on Multi-level Methods (Meisdorf, 1994).

\bibitem{zag06}
{\sc S.~Zaglmayr}, {\em High Order Finite Element Methods for Electromagnetic
  Field Computation}, PhD thesis, Institut f{\"u}r Numerische Mathematik, 2006.

\end{thebibliography}
\bibliographystyle{siam}

\end{document}